\documentclass[10pt]{amsart}
\usepackage{amsxtra,amscd,amsthm,amsmath,amssymb}
\usepackage{longtable} 
\usepackage{enumerate}
\usepackage{verbatim}
\usepackage{graphicx}
\usepackage{xspace}
\usepackage{ifthen}
\usepackage{tabularx}

\usepackage{color}
\usepackage{mathrsfs}

\makeatletter
\renewcommand*\env@matrix[1][*\c@MaxMatrixCols c]{%
  \hskip -\arraycolsep
  \let\@ifnextchar\new@ifnextchar
  \array{#1}}
\makeatother

\usepackage{ mathdots }
\usepackage{flexisym}
\usepackage[utf8]{inputenc}
\usepackage{amsfonts}
\usepackage{multicol}
\usepackage{collectbox}

\usepackage[all, cmtip]{xy}
\usepackage{amsthm,amsfonts,amssymb,amsmath,amsxtra}
\usepackage[all]{xy}
\usepackage{amsmath}
\usepackage{flexisym}
\usepackage{amsfonts}
\usepackage{amsmath}

 \usepackage{hyperref}

\usepackage{multicol}
\usepackage{collectbox}

\usepackage{tikz-cd}

\usepackage{amsthm}

\theoremstyle{plain}
\numberwithin{equation}{subsection}
\newtheorem{thm}[equation]{Theorem}

\newtheorem{prop}[equation]{Proposition}

\newtheorem{Definition}[equation]{Definition}

\newtheorem{Remark}[equation]{Remark}

\theoremstyle{remark}

\theoremstyle{plain}
\renewcommand{\subsubsection}{\addtocounter{equation}{1}{\vskip 6pt \bf\arabic{section}.\arabic{subsection}.\arabic{equation}}}
\theoremstyle{definition}

\newcommand{\quash}[1]{}  
\newcommand{\nc}{\newcommand}
\nc{\on}{\operatorname}

\newcommand{\lps}{[\![}
\newcommand{\rps}{]\!]}
\newcommand{\llps}{(\!(}
\newcommand{\lrps}{)\!)}

\renewcommand{\phi}{\varphi}

\newcommand{\bbA}{{\mathbb A}}

\newcommand{\bbG}{{\mathbb G}}

\newcommand{\bbL}{{\mathbb L}}

\newcommand{\bbP}{{\mathbb P}}

\newcommand{\calA}{{\mathcal A}}
\newcommand{\calB}{{\mathcal B}}

\newcommand{\calG}{{\mathcal G}}

\newcommand{\calL}{{\mathcal L}}
\newcommand{\calM}{{\mathcal M}}

\newcommand{\calV}{{\mathcal V}}

\newcommand{\calX}{{\mathcal X}}
\newcommand{\calY}{{\mathcal Y}}

\nc{\al}{{\alpha}} \nc{\be}{{\beta}}

\nc{\ve}{{\varepsilon}} \nc{\Ga}{{\Gamma}}

\nc{\La}{{\Lambda}}

\def\0{\circ}

\newcommand{\tol}{\longrightarrow}

\newcommand{\br}{\breve}

\newcommand{\Mfr}{{\mathfrak M}}

\newcommand{\C}{{\mathbb C}}
\newcommand{\R}{{\mathbb R}}
\newcommand{\Q}{{\mathbb Q}}

\newcommand{\Hom}{{\rm Hom}}

\newcommand{\Gm}{{{\mathbb G}_{\rm m}}}
\newcommand{\Z}{{\mathbb Z}}

\newcommand{\ti}{\tilde}
\newcommand{\Spec}{{\rm Spec \, } }

 \renewcommand{\O}{{\mathcal O}}

\newcommand{\GL}{{\rm GL}}

\newcommand{\und}{\underline}

\newcommand{\Mloc}{{\rm M}^{\rm loc}}

\newcommand{\Mbl}{{\rm M}^{\rm bl}}

\def\thfill{\null\nobreak\hfill}

\def\endproof{\thfill\vbox{\hrule
  \hbox{\vrule\hbox to 5pt{\vbox to 5pt{\vfil}\hfil}\vrule}\hrule}}

\newcommand{\GSpin}{{\rm GSpin}}
\newcommand{\Spin}{{\rm Spin}}
\newcommand{\SO}{{\rm SO}}

\newcommand{\cQ}{{\mathcal Q}}

\newcommand{\lan}{\langle}
\newcommand{\ran}{\rangle}
\newcommand{\wti}{\widetilde}
\newcommand{\un}{\underline}

\begin{document}

\title[integral models for Shimura varieties]{Regular integral models for Shimura varieties\\ of orthogonal type}

\author[G. Pappas]{G. Pappas}
\email{pappasg@msu.edu}
\address{Dept. of
Mathematics\\
Michigan State
Univ.\\
E. Lansing\\
MI 48824\\
USA}

\thanks{The first author acknowledges support by NSF grant \# DMS-2100743}

\author[I. Zachos]{I. Zachos}
 \email{zachosi@bc.edu}
\address{
Dept. of
Mathematics\\
Boston College\\
Chestnut Hill\\
MA 02467\\
USA}


\begin{abstract}
 We consider Shimura varieties for orthogonal or spin groups acting on hermitian symmetric domains of type IV. We give regular $p$-adic integral models for these varieties over odd primes $p$ at which the level subgroup is the connected stabilizer of a vertex lattice in the orthogonal space. Our construction is obtained by combining 
results of Kisin and the first author with an explicit presentation and resolution of a corresponding local model.

\end{abstract}

\maketitle

\tableofcontents

\section{Introduction}\label{sIntro}

\subsection{} This paper makes a contribution towards the goal of constructing regular integral models for Shimura varieties
over places of bad reduction. Constructing and understanding such well-behaved integral models is an interesting and hard problem whose solution has many applications to number theory. The Shimura varieties we are considering here are of orthogonal type and are
quotients of hermitian symmetric domains of type IV. They are associated to spin or orthogonal groups
for a rational quadratic space $(V, Q)$ of dimension $d$ with real signature 
$(d-2, 2)$.  When $d\geq 7$, these Shimura varieties are not of PEL type and so they cannot be given directly as moduli spaces of abelian varieties with polarization, endomorphisms and level structure. However, they are always either of Hodge or of abelian type. So, they can still be constructed using a relation to Siegel moduli spaces, i.e. to moduli spaces of polarized abelian varieties. 

 Shimura varieties have canonical models over the ``reflex" number field.  In the cases we consider here the reflex field is the field of rational numbers $\Q$.  They are also expected to afford reasonable integral models. However, the behavior of these depends very much on the ``level subgroup". Here, we consider   level subgroups determined by the choice of a lattice $\Lambda\subset V$ on which the quadratic form takes integral values, i.e. for which
\[
\La\subset \La^\vee
\]
where $\La^\vee$ is the dual lattice. 
In fact, we study $p$-adic integral models and their reduction over odd primes $p$ at which the $p$-power part of the discriminant module
$\La^\vee/\La$ is annihilated by $p$. The level subgroup at $p$ is parahoric in the sense of Bruhat-Tits and this allows us to apply the results of the first author with Kisin \cite{KP}
and construct a $p$-adic integral model with controlled singularities.  We then build on this and, by using work of the second author in \cite{Zachos}, we study and resolve the singularities. This leads to regular models for these Shimura varieties
over the $p$-adic integers $\Z_p$. The models we construct have very simple local structure: Their fibers over $p$ are divisors with normal crossings,
with multiplicities one or two, and with no more than three branches intersecting at a point. We expect that our construction will find applications to the study of arithmetic intersections of special cycles and Kudla's program. (See \cite{FGHM} for an important application of integral models of spin/orthogonal Shimura varieties to number theory.)

\subsection{}  Let us give some details. For technical reasons, it is simpler to discuss Shimura varieties for the group $G=\GSpin(V, Q)$
of spinor similitudes. We take $X$ to be the corresponding hermitian symmetric domain of type IV (see \S \ref{sShimura}).
Let $p$ be an odd prime and
 choose a $\Z_p$-lattice $\La\subset V\otimes_{\Q}\Q_p$ with $p\La^\vee\subset \La\subset \La^\vee$.
The lattice defines the parahoric subgroup
\[
K_p=\{g\in \GSpin(V\otimes_{\Q}\Q_p)\ |\ g\La g^{-1}=\La, \ \eta(g)\in \Z^\times_p\}
\]
which we fix below. (Here, $\eta: \GSpin(V\otimes_{\Q}\Q_p)\to \Q_p^\times$ is the spinor similitude character,
and for $v\in V\otimes_{\Q}\Q_p$, $gvg^{-1}$ is defined using the Clifford algebra, see \S \ref{ss23}, \S \ref{ss24}.)
 Choose also a sufficiently small compact open subgroup $K^p$ of the prime-to-$p$ finite adelic points $G({\mathbb A}_{f}^p)$ of $G$
 and set $K=K^pK_p$. 
 The Shimura variety $ {\rm Sh}_{K}(G, X)$ with complex points
 \[
 {\rm Sh}_{K}(G, X)(\C)=G(\Q)\backslash X\times G({\mathbb A}_{f})/K
 \]
 is of Hodge type and has a canonical model over the reflex field $\Q$.  
 The following is a special case of \cite[Theorem 4.2.7]{KP}, see also \cite[Theorem 3.7]{Pappasicm}.

 \begin{thm}\label{LMKP}
 There is a
 scheme $\mathscr S_K(G, X)$, flat over $\Spec(\Z_{p})$, 
 with
 \[
 \mathscr S_K(G, X)\otimes_{\Z_p}\Q_p={\rm Sh}_K(G, X)\otimes_{\Q}\Q_p,
 \]
 and which supports a ``local model diagram"
  \begin{equation}\label{LMdiagramIntro}
\begin{tikzcd}
&\wti{\mathscr{S}}_K(G, X)\arrow[dl, "\pi_K"']\arrow[dr, "q_K"]  & \\
\mathscr S_K(G, X)  &&  \Mloc(\La)
\end{tikzcd}
\end{equation}
such that:
\begin{itemize}
\item[a)] $\pi_K$ is a $\calG$-torsor for the parahoric group scheme $\calG$ that corresponds to $K_p$,

\item[b)] $q_K$ is smooth and $\calG$-equivariant.
\end{itemize}
\end{thm}

In the above, $\calG$ is the smooth connected Bruhat-Tits group scheme over $\Spec(\Z_p)$ such that
\begin{itemize}
\item  $\calG\otimes_{\Z_p}\Q_p=G\otimes_\Q\Q_p$, 
\item  $\calG(\Z_p)=K_p$,
\end{itemize}
 and $\Mloc(\La)$ is the ``local model" as defined by the first author and X. Zhu \cite{PZ}
 (following previous work by Rapoport-Zink \cite{RZbook} and others). 
 
 The integral model $\mathscr S_K(G, X) $ satisfies several additional properties, see \cite{KP} and \S \ref{sShimura}. 
 In fact, it is ``canonical" in the sense of \cite{PaCan}. Set $\delta={\rm length}_{\Z_p}(\La^\vee/\La)$ and $\delta_*={\rm min}(\delta, d-\delta)$.   
 When $\delta_*=0$, then $\La=\La^\vee$ or $\La=\pi\La^\vee$, the parahoric group $K_p$ is hyperspecial,
 and both $\Mloc(\La)$ and $\mathscr S_K(G, X)$ are smooth over $\Z_p$.  This is the case of good reduction studied in general by Kisin in \cite{KisinJAMS}. In what follows, we will exclude the case $\delta_*=0$, which, for the goals of this paper, holds no interest.
 When $\delta_*=1$, both $\Mloc(\La)$ and $\mathscr S_K(G, X)$ are regular, as established by Madapusi Pera \cite{MP} (see also \cite[12.7.2]{HPR}).

Note that, since $\calG$ is smooth, properties (a) and (b) imply that every point of $\mathscr S_K(G, X) $ has an \'etale neighborhood which is also \'etale over the local model $\Mloc(\La)$.  The local model $\Mloc(\La)$ is a flat projective scheme  over $\Spec(\Z_p)$ with $\calG$-action. It is a $\Z_p$-model of the
 compact dual $X^\vee$ of the hermitian domain $X$. Here, the compact dual $X^\vee$ is the quadric hypersuface $\cQ(V)$ in $\bbP^{d-1}$
 which parametrizes isotropic lines $L\subset V$. Hence, the local model $\Mloc(\La)$ is a distinguished $\Z_p$-model of the quadric 
 $\cQ(V)$ obtained from the lattice $\La$. It is a ``$p$-adic degeneration" of the quadric hypersurface $\cQ(V)$.
 
 Note that $\bbP(\La)$ is a distinguished $\Z_p$-model of the projective space $\bbP(V)$. One could hope that $\Mloc(\La)$ is the flat closure of $\cQ(V_{\Q_p})\subset \bbP(V_{\Q_p})$ in $\bbP(\La)$ but this is not true, unless $\delta_*=0$ or $\delta_*=1$. The next naive guess
 is that $\Mloc(\La)$ is isomorphic to  the flat closure of $\cQ(V_{\Q_p})\subset \bbP(V_{\Q_p})\times \bbP(V_{\Q_p})$ in $\bbP(\La)\times \bbP(\La^\vee)$ but this is also not correct (see below). Nevertheless, the relation of the local model $\Mloc(\La)$ with this last flat closure, which we denote by $\cQ(\La, \La^\vee)$ and call a ``linked quadric",
 plays an important role in our discussion. Let us mention here that, local models for Shimura varieties
 of PEL type have been studied using connections to linked Grassmannians and other classical algebraic varieties
 (see \cite{PRS}). The current paper extends this to the first non-trivial non-PEL type case. The plot now is thicker because the Hodge type Shimura cocharacter of the orthogonal group is no longer minuscule in the standard representation. So, one should not really expect
a straightforward embedding of the local model in the standard linked Grassmannian.
 To understand $\Mloc(\La)$ one needs,  either to use a spin representation (which has very high dimension), or apply directly the definition of $\Mloc(\La)$ (\cite{PZ}) via 
 Beilinson-Drinfeld affine Grassmannians. The latter route was taken by the second author in his thesis \cite{Zachos} to obtain an explicit description of a dense open affine chart as we explain below. 
 
 Note that the $\calG$-action of $\Mloc(\La)$ has only a finite number of orbits (parametrized by the $\mu$-admissible set) and there is a unique closed orbit given by a distinguished point $*$ in the special fiber of $\Mloc(\La)$, which we call the ``worst point".
 
 Set $Z=(z_{ij})\in {\rm Mat}_{\delta \times (d-\delta)}$.  
Denote by $\mathscr D^{2}_{\delta\times (d-\delta)}=\{Z\ |\ \wedge^2Z=0\}\subset {\rm Mat}_{\delta \times (d-\delta)}$ the determinantal subscheme (a Segre cone) of the affine space of matrices $Z$ over $\Spec(\Z_p)$ given by the vanishing of all $2\times 2$ minors of $Z$. 
For simplicity, we assume that the quadratic form on $V\otimes_{\Q}\Q_p$ is split, or quasi-split. (We can always reduce to this case after base changing to $\Q_{p^2}$.).  The following result is essentially shown in \cite{Zachos}, but is stated there in a different form:

\begin{thm}\label{ThmexplicitIntro}
There exists an open affine chart $U\subset \Mloc(\La)$ which contains the worst point $*$ and which is isomorphic to the closed  subscheme of the determinantal scheme $\mathscr D^{2}_{\delta\times (d-\delta)}$
given by the quadratic equation
\[
\sum_{1\leq i\leq \delta,\, 1\leq j\leq d-\delta}z_{i\ d-\delta+1-j}\, z_{\delta+1-i\ j}=-4p.
\]
\end{thm}

Observe that the open subscheme $U$ intersects every $\calG$-orbit in $\Mloc(\La)$, since it contains the unique closed $\calG$-orbit $*$. Hence,
$U$ ``captures all the singularities" of $\Mloc(\La)$ and hence also of  $\mathscr S_K(G, X) $. 
When $\delta_*=1$, it is an open affine in the quadric $\cQ(\La)$. The case $\delta_*=1$ was studied in \cite{MP} and \cite{HPR}.
For $\delta_*\geq 2$, the schemes $\Mloc(\La)$ and $U$ are harder to understand. Let us add that, for $\delta_*\neq 0$, none of these models are smooth or semi-stable (as also follows from the result of \cite{HPR}). The level subgroup which is the stabilizer
of a pair of lattices $(\La_0, \La_1)$ with $\pi\La_0\subset \Lambda_1\subset \Lambda_0$ and $\La^\vee_0=\La_0$, $\La_1^\vee=p^{-1}\La_1$, gives a semi-stable model, as was first shown by Faltings \cite{Faltings}, but it is not of the type considered here.

 \subsection{} We can now consider the blow-up of $\Mloc(\La)$ 
 at the worst point $*$. This gives a $\calG$-birational projective morphism
\[
r^{\rm bl}: \Mbl(\La)\tol \Mloc(\La).
\] 
 Using the explicit description of $U$ above, we show:

\begin{thm}\label{BUpIntro}  The scheme $\Mbl(\La)$ is regular and has special fiber a divisor with normal crossings. In fact, $\Mbl(\La)$ is covered by open subschemes which are smooth over $
\Spec(\Z_p[u, x, y]/(u^2xy-p))
$
when $\delta_*\geq 2$, or over
$
\Spec(\Z_p[u, x]/(u^2x-p))
$
when $\delta_*= 1$. 
\end{thm}

We   quickly see that the corresponding blow-up ${\mathscr{S}}^{\rm reg}_K(G, X)$ of the integral model $\mathscr S_K(G, X) $
inherits the same nice properties as $\Mbl(\La)$. In fact, there is a local model diagram for ${\mathscr{S}}^{\rm reg}_K(G, X)$ similar to (\ref{LMdiagramIntro})  
but with $\Mloc(\La)$ replaced by $\Mbl(\La)$. See Theorem \ref{LM} for the precise statement about the model $\mathscr S^{\rm reg}_K(G, X) $; this theorem is the main result of the paper. The construction of ${\mathscr{S}}^{\rm reg}_K(G, X)$ from $r^{\rm bl}$ and the
local model diagram (\ref{LMdiagramIntro}) is an example of a ``linear modification" in the sense of \cite{PaUnitary}.

\subsection{} To understand the full picture,
we also give two alternative descriptions of the resolution $r^{\rm bl}: \Mbl(\La)\to \Mloc(\La)$: 

The first is partially moduli-theoretic and is inspired by a classical idea: The determinantal scheme
$\mathscr D^2_{n\times m}=\{A\in {\rm Mat}_{n\times m}\ |\ \wedge^{2} A=0\}$ can be resolved by considering
\[
\wti {\mathscr D}^2_{n\times m}=\{(A, L)\in {\rm Mat}_{n\times m}\times \bbP^{m-1}\ |\ \wedge^{2} A=0,\ {\rm Im}(A)\subset L\}
\]
which maps birationally to $\mathscr D^2_{n\times m}$ by the forgetful map $(A, L)\mapsto A$.
This leads to  the definition of a $\calG$-equivariant birational morphism
\[
r: \calM(\La)\tol \Mloc(\La)
\]
which we then show agrees with the blow-up $r^{\rm bl}$ via a $\calG$-equivariant isomorphism 
\begin{equation}\label{iso1}
\calM(\La)\simeq \Mbl(\La).
\end{equation}

The second description relates $\Mloc(\La)$ to the classical linked quadric $\cQ(\La, \La^\vee)$. We show that there is a
blow-up 
$\wti\cQ(\La, \La^\vee)\tol \cQ(\La, \La^\vee)$ along an irreducible component  of the special fiber of $\cQ(\La, \La^\vee)$
(this component is a divisor which is not Cartier) and a $\calG$-equivariant isomorphism
\begin{equation}\label{iso2}
\wti\cQ(\La, \La^\vee)\simeq \Mbl(\La).
\end{equation}
The proofs of the two isomorphisms (\ref{iso1})  and (\ref{iso2}) are intertwined.
As a result, we have a diagram of $\calG$-equivariant birational projective morphisms
 \begin{equation}\label{eqDiagramIntro}
\begin{tikzcd}
&\Mbl(\La) \arrow[dl, "r^{\rm bl}"'] \arrow[dr, "\rho"] &\\
\Mloc(\La) &&  \cQ(\La, \La^\vee)
\end{tikzcd}
\end{equation}
and we can pass from the linked quadric $\cQ(\La, \La^\vee)$ to $\Mloc(\La)$ as follows:   
We first blow-up along an irreducible component $Z_1$  (a non-Cartier divisor, if $\delta_*\geq 2$)   of the special fiber of $\cQ(\La, \La^\vee)$. Then we blow the proper transform $\bbP^{\delta-1}\times {\bbP}^{d-\delta-1}=\wti Z_0\subset \wti\cQ(\La, \La^\vee)$ of another irreducible component $Z_0$ down to the point $*$, to obtain $\Mloc(\La)$. 

Let us mention here that families of degenerating quadrics, like the schemes $\Mloc(\La)$ and $\cQ(\La, \La^\vee)$, are objects of the old theory of complete quadrics (see \cite{CompleteQ} and the references there). An interesting problem is to reinterpret our constructions within the framework of this classical theory.

\subsection{} In the above, we only discussed Shimura varieties for $G=\GSpin(V)$. However, the results also apply to $G=\SO(V)$ (see \S \ref{ss62b}), and can be used for other related groups. For example, we can consider groups of the form ${\rm Res}_{F/\Q}(G_F)$, where $G_F=\GSpin(V_F)$, or $G=\SO(V_F)$, for a quadratic space $V_F$ over the number field $F$, provided $F$ is unramified over $p$.
In that case, the local models over $p$ are given as $d$-fold products of the local models above and we can construct integral models
of Shimura varieties by using the products of the corresponding resolutions. After certain explicit blow-ups of these products, we can also obtain regular integral models. Similarly, we can apply these results to obtain regular (formal) models of the corresponding Rapoport-Zink spaces constructed in  \cite[\S 5]{HamaKim}.
\medskip

{\bf Acknowledgements:}
We thank B. Howard, M. Rapoport, and the referee, for useful suggestions.

\section{Preliminaries: quadratic forms, lattices and spinors}\label{s2}

\subsection{Quadratic forms and lattices}\label{ss21} Let us fix an odd prime $p$ and consider a finite field extension $F/\Q_p$. 
Let $\O=\O_F$ be the ring of integers
and let $\pi$ be a uniformizer of $\O$. We will denote by $k=\O/(\pi)$ the residue field and by $\bar k$ 
an algebraic closure of $k$. Denote also by $\br F$ the completion of the maximal unramified extension of $F$ in an algebraic closure $\bar F$ and by $\br O$ the integers of $\br F$. We can assume $\bar k=\br\O/(\pi)$. 

Let $V$ be an $F$-vector space of dimension $d$ and fix
\[
\lan\ ,\ \ran: V \times V\to F,
\]
a non-degenerate symmetric $F$-bilinear form. 
Set
\[
Q(v)=\frac{1}{2}\lan v, v\ran,
\]
for the associated quadratic form $Q: V\to F$.
We assume that $d\geq 5$ throughout. 

If $\La\subset V$ is an $\O$-lattice, we set
\[
\Lambda^\vee=\{v\in V\ |\ \lan v, a\ran\in \O,\ \forall a\in \Lambda\}
\]
for the dual of $\Lambda$. In what follows, we will   fix a ``vertex lattice" $\Lambda$ in $V$. By definition this is an 
$\O$-lattice $\La$ such that 
\[
\pi\Lambda^\vee\subset \Lambda\subset \Lambda^\vee.
\]
Set $\delta(\La):={\rm length}_\O(\Lambda^\vee/\Lambda)=\dim_k(\Lambda^\vee/\Lambda)$, $\delta_*={\rm min}(\delta, d-\delta)$.
We will often assume $\delta(\Lambda)\leq d/2$. Indeed, we can always
replace the form by a multiple and exchange the roles of $\La$ and $\La^\vee$ to come to this situation.
Then $\delta_*=\delta$.

The form $\lan\ ,\ \ran$ induces symmetric $\O$-bilinear forms 
\[
\lan\ ,\ \ran: \Lambda\times\Lambda \to \O, \quad \pi\lan\ ,\ \ran : \Lambda^\vee\times \Lambda^\vee \to \O.
\]
We also obtain perfect symmetric $k$-bilinear forms
\[
\lan\, ,\, \ran_1: \Lambda/\pi\Lambda^\vee \times \Lambda/\pi\Lambda^\vee \to k, \qquad 
\lan\, ,\, \ran_2: \Lambda^\vee/\Lambda \times \Lambda^\vee/\Lambda \to k,
\]
given by $\lan x, y\ran_1=\lan x, y \ran\, {\rm mod}\, (\pi)$, $\lan x, y\ran_2=\pi\lan x, y \ran\, {\rm mod}\, (\pi)$.

\subsection{Normal forms}\label{ss42a}

As in  \cite[Appendix, Lemma A.25]{RZbook}, 
we can write 
\[
\La=M\oplus N, \quad \La^\vee=M\oplus \pi^{-1}N,
\]
where $M$, $N$ are free $\O_{F}$-submodules such that $\lan M, N\ran= 0$ and with the property that the form $\lan\ ,\ \ran$ is perfect on
$M$ and $\pi^{-1}\lan\ ,\ \ran$ is perfect on $N$. In fact, by \emph{loc. cit.} Prop. A.21, after a finite unramified base change $F'/F$ we can find an $\O_{F'}$-basis $\{e_i\}$ of $\La\otimes_{\O_F}\O_{F'}$, i.e. write
\[
\La_{\O_{F'}}= \Lambda\otimes_{\O_F}\O_{F'} =\bigoplus_{i=1}^d \O_{F'}\cdot e_i
 \]
such that exactly one of the following cases occurs:
\medskip

    (1) $d=2n$, $\delta(\Lambda)=2r$:  
  \[
    \La_{\O_{F'}}=(e_1, \ldots, e_{n-r})\oplus (e_{n-r+1},\ldots , e_{n+r})\oplus (e_{n+r+1}, \ldots, e_d),
\]
\[
M=(e_1, \ldots, e_{n-r})\oplus (e_{n+r+1}, \ldots, e_d),\quad
 N=(e_{n-r+1},\ldots , e_{n+r}),
 \]
 \begin{eqnarray*}
 \hbox{\rm with\ \ \ }   \langle e_i , e_{d+1-j} \rangle\! \!&=& \delta_{ij}, \ \ \hbox{\rm for\ }\ i<n-r+1, \ \ \hbox{\rm or\ } \ n+r<i, \\
    \langle e_i , e_{d+1-j} \rangle\! \!&=& \pi \delta_{ij}, \ \hbox{\rm for\ }\ n-r+1\leq i\leq  n+r. 
    \end{eqnarray*}
    
    \smallskip
    
    (2)  $d=2n$, $\delta(\Lambda)=2r+1$:  
  
    \[
\La_{\O_{F'}}=(e_1, \ldots, e_{n-r-1})\oplus  (e_{n-r},\ldots , e_{n-1})\oplus (e_n, e_{n+1})\oplus (e_{n+2},\ldots , e_{n+r+1})\oplus (e_{n+r+2}, \ldots, e_d),
\]
\[
M=(e_1, \ldots, e_{n-r-1})\oplus (e_{n+1})\oplus (e_{n+r+2}, \ldots, e_d),
\]
\[
N=(e_{n-r},\ldots , e_{n-1})\oplus (e_n)\oplus (e_{n+2},\ldots , e_{n+r+1}),
\]
  \begin{eqnarray*}
    \hbox{\rm with\ \ \ }    \langle e_i , e_{d+1-j} \rangle\! \!&=& \delta_{ij}, \ \ \hbox{\rm for\ }\  i< n-r,\ \hbox{\rm or} \ n+r+1<i, \\
    \langle e_i , e_{d+1-j} \rangle &=& \pi \delta_{ij}, \ \ \hbox{\rm for\ }\ n-r\leq i\leq n+r+1,\ i\neq n,\ i\neq n+1, \\
    \ \langle e_n ,  e_n \rangle &=& \pi, \quad \langle e_{n+1} ,  e_{n+1} \rangle = 1, \quad  \langle e_n , e_{n+1} \rangle = 0.
    \end{eqnarray*}
\smallskip

    (3)   $d= 2n+1$, $\delta(\Lambda)=2r$:  
   
    \[
\La_{\O_{F'}}=(e_1, \ldots, e_{n-r})\oplus  (e_{n-r+1},\ldots , e_{n})\oplus (e_{n+1})\oplus (e_{n+2},\ldots , e_{n+r+1})\oplus (e_{n+r+2}, \ldots, e_d),
\]
\[
M=(e_1, \ldots, e_{n-r})\oplus (e_{n+1})\oplus (e_{n+r+2}, \ldots, e_d),
\]
\[
N=(e_{n-r+1},\ldots , e_{n})\oplus (e_{n+2},\ldots , e_{n+r+1}),
\]
  \begin{eqnarray*}
  \hbox{\rm with\ \ \ }      \langle e_i , e_{d+1-j} \rangle &=& \delta_{ij}, \ \ \hbox{\rm for\ }\ i<n+1-r, \
    \hbox{\rm or}\ i=n+1, \ \hbox{\rm or}\ n+1+r<i, \\
    \ \langle e_i , e_{d+1-j} \rangle &=& \pi \delta_{ij}, \ \ \hbox{\rm for\ }\ n+1-r\leq i\leq n+1+r,\ \hbox{\rm and\ }\ i\neq n+1.
    \end{eqnarray*}
\smallskip

 (4)  $d= 2n+1$, $\delta(\Lambda)=2r+1$:  
  
    \[
\La_{\O_{F'}}=(e_1, \ldots, e_{n-r})\oplus (e_{n-r+1},\ldots , e_{n+r+1})\oplus (e_{n+r+2}, \ldots, e_d),
\]
\[
M=(e_1, \ldots, e_{n-r})\oplus (e_{n+r+2}, \ldots, e_d),\quad
N=(e_{n-r+1},\ldots , e_{n+r+1}),
\]
  \begin{eqnarray*}
  \hbox{\rm with\ \ \ }      \langle e_i , e_{d+1-j} \rangle &=& \delta_{ij}, \ \ \hbox{\rm for\ }\ i<n+1-r, \ \hbox{\rm or}\ n+1+r <i, \\
    \langle e_i , e_{d+1-j}\ran &=& \pi \delta_{ij}, \ \ \hbox{\rm for\ }\ n+1-r\leq i\leq  n+1+r .
    \end{eqnarray*} 
\smallskip

 In all the above, the parentheses give a short-hand notation for the $\O_{F'}$-lattice generated by the included vectors.
For simplicity,  we omit the notation of the base change of $M$, $N$.
In all cases, we will denote by $S$ the (symmetric) matrix with entries $\lan e_i, e_j\ran$ where $\{e_i\}$ is the basis above.
We can then write 
\[
S=S_1+\pi S_2
\]
where $S_1$, $S_2$ both have entries only $0$ or $1$. For example, in case (1) we write:
\[
S_1 := \begin{pmatrix} 
          &  & 1^{(n-r)}  \\
          & 0^{(2r)}  & \\
         1^{(n-r)} &  &   \\
         
    \end{pmatrix},\quad     S_2:=\begin{pmatrix} 
          &  & 0^{(n-r)}  \\
          & 1^{(2r)}  & \\
         0^{(n-r)} &  &   \\
        \end{pmatrix}.
        \]

\subsection{Spinor groups}\label{ss23}

The Clifford algebra  of $V$ is a $\Z/2\Z$-graded $F$-algebra denoted
\[
C(V)=C^+(V) \oplus C^- (V). 
\]
It is a vector space of rank $2^d$ over $F$, generated as an algebra by the image of a canonical injection
$V\hookrightarrow C^-(V)$ satisfying $v\cdot v=Q(v)$.  The {canonical involution} on $C(V)$ is the    
$F$-linear endomorphism $c\mapsto c^*$ characterized by 
 $
 (v_1 \cdots v_m)^*=v_m\cdots v_1,
 $
 for $v_1,\ldots, v_m\in V$.

For an $F$-algebra $R$,  the tensor product   $V_R=V\otimes_{F} R$ is a nondegenerate quadratic space over $R$ 
with  Clifford algebra  $C(V_R)=C(V)\otimes_{F}R$.
  The {spinor similitude group}   $G=\GSpin(V)$   is the reductive group over $F$ with $R$-points
\[
\GSpin(V)(R) = \{ g\in C^+(V_R)^\times : g V_R g^{-1} =V_R,\, g^*g \in R^\times \},
\]
 and the {spinor similitude}  
 $
 \eta :\GSpin(V) \to \Gm
 $ 
 is the character  $\eta (g) = g^* g$.

The conjugation action of $G$ on $C(V)$ leaves invariant the $F$-submodule $V$, and this action of $G$ on $V$ is denoted
$g\cdot v = gvg^{-1}$.  There is a short exact sequence of reductive group schemes
\begin{equation}\label{ces1}
1 \to \Gm \to G \xrightarrow{g\mapsto g \cdot} \SO(V) \to 1
\end{equation}
over $F$, and the restriction of $\eta$  to the central $\Gm$ is  $z\mapsto z^2$.

 The spinor group $\Spin(V)$ is the kernel of $\eta$ and there is a short exact sequence
 of reductive groups
 \begin{equation}\label{ces2}
 1\to \Spin(V)\to G=\GSpin(V)\xrightarrow{\eta} \Gm\to 1.
 \end{equation}
 We have $G_{\rm der}=\Spin(V)$ which is simply connected. Hence, $\pi_1(G)=\pi_1(\Gm)\simeq \Z$
 with trivial Galois action. In particular, we have 
 $
 \pi_1(G)_I\simeq \Z
 $
 for the coinvariants of the action by the inertia group $I={\rm Gal}(\bar F/\breve F)$. Recall, $\pi_1(G)_I$ is the target
 of the Kottwitz homomorphism
 \[
\kappa:  G(\breve F)\to \pi_1(G)_{I}.
 \]
 Hence, in our case, the Kottwitz homomorphism is
 \[
 \kappa: \GSpin(V)(\breve F)\to \Z.
 \]
 In fact, by the definition of $\kappa$ and (\ref{ces2}), we can see that 
 \begin{equation}\label{kot}
 \kappa(g)={\rm val}(\eta(g)),
 \end{equation}
  i.e. $\kappa$ is given by the valuation of the spinor similitude.

\subsection{Orthogonal parahoric groups}\label{ss22} We first consider the special orthogonal group $\SO(V)$.
The self-dual lattice chain 
\[
\cdots\subset \pi\La^\vee\subset \La\subset \La^\vee\subset \pi^{-1}\La\subset\cdots
\] gives a point $x_\La$ of the Bruhat-Tits building $\calB(\SO(V), F)$ (see for example \cite{BTIV}). Let us consider the subgroup 
\[
K=K_\La=\{g\in \SO(V)\ |\ g(\Lambda)=\Lambda\} \subset \SO(V)
\]
which preserves the lattice. Then every element $g\in K$ also preserves the dual lattice, $g(\La^\vee)=\La^\vee$,
and so it induces elements $\bar g_1\in \GL(\Lambda/\pi\Lambda^\vee)$, $\bar g_2\in \GL(\Lambda^\vee/\Lambda)$,
which lie in the corresponding orthogonal groups for the forms $\lan\ ,\ \ran_1$, $\lan\ ,\ \ran_2$.
We can see that $g\mapsto (\bar g_1, \bar g_2)$ gives a group homomorphism 
\[
K\to {\rm O}(\Lambda/\pi\Lambda^\vee)\times {\rm O}(\Lambda^\vee/\Lambda).
\]
Composing with 
\[
\det\times \det: {\rm O}(\Lambda/\pi\Lambda^\vee)\times {\rm O}(\Lambda^\vee/\Lambda)\to \{\pm 1\}\times \{\pm 1\}
\]
gives $\varepsilon: K\to \{\pm 1\}\times \{\pm 1\}$. The corresponding parahoric subgroup $K^\circ\subset K$ is the kernel of $\varepsilon$.
(See \cite{BTIV} and  \cite[Example 3.12]{Tits}. The map $\varepsilon$ can also be related to the Kottwitz homomorphism for $\SO(V)$
which is given by the spinor norm.) 

Set $i: \Lambda\to \Lambda^\vee$, $j: \pi\Lambda^\vee\to \Lambda$ for the natural $\O$-linear inclusions. We have 
$i\cdot j=\pi$. 

Let $R$ be an $\O$-algebra. For simplicity, set $\La_R=\Lambda\otimes_\O R$, $\La^\vee_R=\La^\vee\otimes_\O R$.
We identify
\[
\La^\vee_R={\rm Hom}_R(\La_R, R)
\]
using the form $\lan\ ,\ \ran_R=\lan\ ,\ \ran\otimes_\O R$. 
Consider the group scheme $\calG=\calG_\La$ over $\Spec(\O)$ 
which  has $R$-valued points given by
$
g \in \GL(\Lambda_R) 
$
for which 
\begin{equation}\label{groupScheme}
\begin{matrix} \pi\La^\vee_R &\xrightarrow{j_R} &\La_R&\xrightarrow{i_R} & \La^\vee_R\\
\pi g^\vee\downarrow\ \ \ \ && g\downarrow &&\ \ \ \downarrow g^\vee\\
 \pi\La^\vee_R &\xrightarrow{j_R} &\La_R&\xrightarrow{i_R} & \La^\vee_R
\end{matrix}
\end{equation}
commutes, and with $\det(g)=1$, $\det(g^\vee)=1$. In the above, $g^\vee $ denotes the $R$-linear map which is $R$-dual to 
$g: \La_R\to \La_R$.  

As we can see, using Appendix \cite{RZbook}, the group $\calG$ is smooth and has $\O$-points given by $K$. 
The homomorphism $\varepsilon$ above is the composition  
\[
K=\calG(\O)\to \calG(k)\to \{\pm 1\}\times \{\pm 1\}
\]
and the kernel gives the neutral component $\calG^\circ$ of $\calG$.
As in \cite{BTIV}, $\calG=\calG_\La$, $\calG^\circ=\calG^\circ_\La$, are the Bruhat-Tits group schemes that
correspond to $x_\La$, or, in other words, to $K$, resp. $K^\circ$.
 
 \subsection{Spinor parahoric groups}\label{ss24}
 
 Consider the (extended) Bruhat-Tits building $\calB^e(G, F)$ for $G=\GSpin(V)$ over $F$. The central exact sequence (\ref{ces1}) induces by \cite[(4.2.15)]{BTII}, or \cite[Theorem 2.1.8]{La}, a canonical $G(F)$-equivariant map
 \[
 \calB^e(G, F)\to \calB^e(\SO(V), F)=\calB(\SO(V), F)
 \]
 which lifts an identification
 \[
 \calB(G, F)\xrightarrow{\sim} \calB(\SO(V), F)
 \]
  between the (classical) Bruhat-Tits buildings. Consider now $x\in \calB^e(G, F)$ which we assume maps 
  to the point $x_\La\in \calB(G, F)\simeq \calB(\SO(V), F)$ defined by a vertex lattice $\La\subset V$. Then, the stabilizer of $x$ is 
  the subgroup 
  \[
  G(F)_x=\{g\in \GSpin(V)(F)\ |\  g\Lambda g^{-1}=\Lambda,\   \eta(g)\in \O^\times\}.
  \]
 By (\ref{kot}), $G(\breve F)_x\subset {\rm ker}(\kappa)$. Hence, by \cite[Appendix by Haines-Rapoport]{PR}, the corresponding Bruhat-Tits group scheme $\calG_x$  is connected, i.e. $\calG_x=\calG^\circ_x$, and it is the corresponding parahoric group scheme over $\O$
  for $\GSpin(V)$ over $F$. By \cite[Prop. 1.1.4]{KP}, we have an exact sequence
  \begin{equation}\label{ces3}
  1\to \Gm\to \calG_x \to \calG_\Lambda^\circ\to 1
  \end{equation}
  of group schemes over $\Spec(\O)$ which extends (\ref{ces1}).

  \begin{Remark}
  {\rm Assume that we are in one of the cases (1)-(4) listed in \S \ref{ss42a}, in particular we assume $F=F'$ in the notations of \S \ref{ss42a}.  Using the identification $
 \calB(\GSpin(V), F)\xrightarrow{\sim} \calB(\SO(V), F)$ and
 the description of the building for $\calB(\SO(V), F)$ given in \cite{BTIV}, we see that each {maximal} compact subgroup of $\GSpin(V)$
  is of the form $G(F)_x$ for $x=x_\La$ given, as above, by some vertex lattice $\La$ with $\delta_*(\La)\neq 2$.
  As we mentioned above, when $\delta_*(\La)=0$, $G(F)_x$ is hyperspecial. When $\delta_*(\La)=1$,
$G(F)_x$ is special. 
  The vertex lattices $\La$ with
  $\delta_*(\La)=2$, i.e. $\delta(\La)=2$ or $\delta(\La)=d-2$, give parahoric subgroups which are not maximal. Indeed, when $\delta(\La)=2$, there are exactly two self-dual
  lattices $\Lambda_0$, $\Lambda'_0$, that fit in an oriflamme configuration:
  \[\xy
	(-9,7)*+{\La};
	(-4.5,3.5)*+{\rotatebox{-45}{$\,\, \subset\,\,$}};
	(-4.5,10.5)*+{\rotatebox{45}{$\,\, \subset\,\,$}};
	(0,14)*+{\Lambda_0};
	(0,0)*+{\Lambda'_0};
	(4,3.5)*+{\rotatebox{45}{$\,\, \subset\,\,$}};
	(4,10.5)*+{\rotatebox{-45}{$\,\, \subset\,\,$}};
	(9,7)*+{ \La^\vee };
	\endxy    
	\]
Then $G(F)_x$ is the intersection of the two hyperspecial subgroups $\GSpin(\La_0)$, $\GSpin(\La'_0)$.
If $\delta(\La)=d-2$, then there are exactly two lattices   $\La_1$ and $\La_1'$ with $\La_1^\vee=\pi\La_1$, ${\La'}_1^\vee=\pi\La'_1$, 
that fit in a similar oriflamme configuration between $\La^\vee$ and $ \pi^{-1}\La$.
\quash{ such that
 \[\xy
	(-9,7)*+{\La^\vee};
	(-4.5,3.5)*+{\rotatebox{-45}{$\,\, \subset\,\,$}};
	(-4.5,10.5)*+{\rotatebox{45}{$\,\, \subset\,\,$}};
	(0,14)*+{\Lambda_1};
	(0,0)*+{\Lambda'_1};
	(4,3.5)*+{\rotatebox{45}{$\,\, \subset\,\,$}};
	(4,10.5)*+{\rotatebox{-45}{$\,\, \subset\,\,$}};
	(10,7)*+{\ \  \pi^{-1} \La. };
	\endxy    
	\]}
The group $G(F)_x$ is again the intersection of the two hyperspecial subgroups
which stabilize $\La_1$ and $\La_1'$.  
  }
  \end{Remark}

\section{Quadrics and linked quadrics}\label{sQua}

\subsection{Quadrics}\label{ss31a}
Suppose that $R$ is an $\O$-algebra and that $(W, \lan\ ,\ \ran)$ is a free rank $d$ 
$R$-module with a symmetric $R$-bilinear form 
\[
\lan\ ,\ \ran: W\times W\to R.
\]
We denote by $\cQ(W, \lan\ ,\ \ran)$, or simply by $\cQ(W)\subset \bbP(W)\simeq \bbP^{d-1}_R$, when the form is understood, the closed subscheme parametrizing isotropic $R$-lines $L$ (i.e. locally $R$-free
direct summands of $W$ of rank $1$) with $\lan L,L\ran=0$.

\subsection{Linked quadrics}\label{ss31}
Suppose now that $V$ and $\La\subset V$ are as in \S \ref{ss21}. We assume  
\[
\delta={\rm length}_\O(\La^\vee/\La)\leq d/2.
\]

Consider the $\O$-scheme 
$P(\La, \La^\vee)$ parametrizing linked isotropic lines 
\[
(L, L')\in \cQ(\Lambda)\times_\O \cQ(\Lambda^\vee)\subset \bbP(\Lambda)\times_\O \bbP(\Lambda^\vee)
\]
where the ``linked" condition for an $R$-valued point $(L, L')$ is that
\[
j_R(\pi L')\subset L,\quad i_R(L)\subset L',
\]
i.e. we have a commutative diagram
\[
\begin{matrix}
\pi\La^\vee_R &\xrightarrow{j_R} &\La_R&\xrightarrow{i_R} & \La^\vee_R\\
\cup && \cup && \cup \\
\pi L' &\xrightarrow{ \ }& L &\xrightarrow{ \ }& L'.
\end{matrix}
\]
(Of course, isotropic means that we require $\lan L, L\ran_R=0$, $\pi \lan L', L' \ran_R =0$.)
There is a natural isomorphism
\[
P(\La, \La^\vee)\otimes_\O F\cong \cQ(V),
\]
since, for an $F$-algebra $R$, we have $\La_R=\La^\vee_R=V_R$.

We  set $\cQ(\Lambda, \Lambda^\vee)$ for the (flat) reduced closure of the 
generic fiber $\cQ(V)$ in the $\O$-scheme $P(\La, \La^\vee)$.
By (\ref{groupScheme}), the group scheme $\calG$ acts on $P(\Lambda, \La^\vee)$ and also 
on the closure $\cQ(\Lambda,\Lambda^\vee)$. If $\delta=0$, then it is easy to see that 
$\cQ(\La,\La^\vee)=\cQ(\La)$ is a smooth quadric hypersurface in $\bbP^{d-1}_\O$.
Part (A) of the following theorem was suggested to the authors by B. Howard.

\begin{thm}\label{thmQ}   Assume $\delta>0$.

A) The $R$-valued points of the scheme $\cQ(\Lambda, \Lambda^\vee)$ are in bijection with the set of pairs $(L, L')$ of $R$-lines $L\subset  \Lambda_R$, $L'\subset \Lambda^\vee_R$, such that
$
j_R(\pi L')\subset L$, $ i_R(L)\subset L',
$ and
\[
\lan L, L'\ran_R=0,
\]
for   the perfect $R$-bilinear pairing $\lan\  ,\,  \ran_R: \Lambda_R\times \Lambda^\vee_R\to R$ induced by $\lan\  ,\,  \ran$.

B) The scheme $\cQ(\Lambda, \Lambda^\vee)$ is normal and is a relative local complete intersection, flat and projective over $\Spec(\O)$.
Its generic fiber is the smooth quadric $\cQ(V)$ in $\bbP(V)$. The scheme is regular if and only if $\delta= 1$.

C) The reduced special fiber of $\cQ(\Lambda, \Lambda^\vee)$ is the union of  $3$ reduced subschemes
$Z_0$, $Z_1$ and $Z_2$ of $\cQ(\Lambda, \Lambda^\vee)_k$, defined as follows:
\begin{itemize}
\item $Z_0$ is the locus of $(L, L')$ for which $i_R(L)=0$ and $j_R(\pi L')=0$.

\item $Z_1$ is the locus of $(L, L')$ for which $L'\subset   \Lambda^\vee_R=(\Lambda^\vee/\pi\Lambda^\vee)\otimes _kR$ lies in the $R$-submodule $(\Lambda /\pi\Lambda^\vee)_R\subset  \Lambda^\vee_R$, and is such that $\lan L'  ,\, L' \ran_{1, R}=0$.

\item $Z_2$ is the locus of $(L, L')$ for which $\pi^{-1}L\subset (\pi^{-1}\Lambda)_R=(\pi^{-1}\Lambda / \Lambda )\otimes _kR$ lies in the $R$-submodule $ ( \Lambda^\vee/  \Lambda)_R\subset (\pi^{-1}\Lambda)_R$, and is such that $\lan \pi^{-1} L ,\,  \pi^{-1}L \ran_{2, R}=0$.
\end{itemize}
These have the following properties:
\begin{itemize}
\item[1)] The subschemes $Z_1$, $Z_2$ are projective space bundles over the quadric hypersurfaces\\  $\cQ(\Lambda/\pi\Lambda^\vee,
\lan\, ,\, \ran_1)\subset \bbP(\Lambda/\pi\Lambda^\vee)$ and $\cQ(\Lambda^\vee/\Lambda, \lan\, ,\, \ran_2)\subset \bbP(\Lambda^\vee/\Lambda)$ respectively.

\item[2)] $Z_0\simeq \bbP(\Lambda^\vee/\Lambda)\times \bbP(\Lambda/\pi\Lambda^\vee)$ and
we have
\[
Z_0\cap Z_1\simeq \cQ(\Lambda/\pi\Lambda^\vee,
\lan\, ,\, \ran_1)\times \bbP(\Lambda^\vee/\Lambda),
\]
\[
 Z_0\cap Z_2\simeq \bbP(\Lambda/\pi\Lambda^\vee)\times \cQ(\Lambda^\vee/\Lambda, \lan\, ,\, \ran_2),
\]
for the scheme theoretic intersections. 

\item[3)] $div(\pi)=2\cdot (Z_0)+(Z_1)+(Z_2)$ as  Weil divisors on $\cQ(\La, \La^\vee)$. 
 
\end{itemize}

If  $\delta>2$, then $Z_0$, $Z_1$ and $Z_2$ are smooth and irreducible.  If $\delta=2$, then $Z_0$, $Z_1$ are smooth and irreducible
but $Z_2$ is either irreducible or the disjoint union of two smooth irreducible components. Finally, if $\delta=1$, then $\cQ(\Lambda^\vee/\Lambda, \lan\, ,\, \ran_2)$ and $Z_2$ are empty
and $Z_0$, $Z_1$ are smooth and irreducible.  
\end{thm}
\begin{proof}
Write
\[
\Lambda=M\oplus N,\quad \Lambda^\vee=M\oplus \pi^{-1}N
\]
with $M$, $N$ free $\O$-submodules of $\La$ such that $\lan M, N\ran=0$ and with the form $\lan\ ,\ \ran$ perfect on $M$ and
$\pi^{-1}\lan\ ,\ \ran$  perfect on $N$. We have ${\rm rank}_\O(N)=\delta$, ${\rm rank}_\O(M)=d-\delta$. 
Note that
\[
\bar M:=M\otimes_\O k=\La/\pi\La^\vee, \quad \bar N:=N\otimes_\O k\simeq \La^\vee/\La.
\]
Set 
\[
\lan\ ,\ \ran_1 : M\times M\to \O,\quad \lan m, m' \ran_1=\lan m, m'\ran.
\]
\[
\lan\ ,\ \ran_2 : N\times N\to \O,\quad \lan n, n' \ran_2=\pi^{-1}\lan n, n'\ran.
\]
Now write
\[
L=(f)=(m+n),\quad L'=(f')=(m'+\pi^{-1}n')
\]
with $m$, $m'\in M_R$ and $n$, $n'\in N_R$.  
The linking conditions are 
\[
m+n=u(m'+\pi^{-1}n'),\quad \pi m'+n'=\pi(m'+\pi^{-1}n')=v(m+n),
\]
for some $u$, $v\in R$. These give $m=u m'$, $n'=vn$, $uv=\pi$.

The isotropic conditions are
\[
\lan f, f\ran=\lan m, m\ran_1+\pi\lan n , n\ran_2=0,
\quad
\pi\lan f', f'\ran =\pi\lan m', m'\ran_1+\lan n', n'\ran_2=0.
\]
These translate to, respectively, 
\[
 u^2\lan m', m'\ran_1+uv\lan n , n\ran_2=0,
\quad
 uv\lan m', m'\ran_1+v^2\lan n, n\ran_2=0.
\]
In the flat closure, they amount to a single equation:
\begin{equation}\label{singleEq}
u\lan m', m'\ran_1+v\lan n , n\ran_2=0.
\end{equation}
By the above, we have $L=(um'+n)$, $L'=(m'+\pi^{-1}vn)$, and so (\ref{singleEq}) amounts to 
\[
\lan L, L'\ran_R=0.
\]

Denote by $\cQ^\sharp(\Lambda, \Lambda^\vee)$ the closed subscheme of $\bbP(\Lambda)\times_\O \bbP(\Lambda')$ given by pairs $(L, L')$ such that
$j_R(\pi L')\subset L$, $ i_R(L)\subset L',
$ and $\lan L, L'\ran_R=0$. By the above discussion, the flat closure $\cQ(\Lambda, \Lambda^\vee)$ is 
a closed subscheme of $\cQ^\sharp(\Lambda, \Lambda^\vee)$ with the same generic fiber. 
To show (A), it is enough to show that $\cQ^\sharp(\Lambda, \Lambda^\vee)$ is flat over $\O$.

Consider the affine scheme over $\Spec(\O)$ given as the product
\[
\calX:= \bbA(M)\times_\O \bbA(N)\times_\O \Spec(\O[u, v]/(uv-\pi))=\Spec(\O[\un  x, \un y, u, v]/(uv-\pi)).
\]
(Here, we set $\bbA(M)=\Spec({\rm Sym}^\bullet_\O(M^\vee))$ and similarly for $\bbA(N)$. We identify $\bbA(M)$, $\bbA(N)$ with $\Spec(\O[\un x])$, $\Spec(\O[\un y])$, after picking a basis of $M$, $N$ respectively.)
Let $\calX^0=\calX-(T_1\cup T_2\cup T_0)$ be the open subscheme of $\calX$ which is the complement
of the union of the three subschemes $T_1$, $T_2$, $T_0$, defined respectively by the ideals
\[
(\un x, v),\quad (\un y, u), \quad (\un x, \un y).
\]
The scheme $\calX^0$ supports an action of the torus $\bbG_m\times\bbG_m$ given by
 \[
 (\lambda, \mu)\cdot (\un x, \un y, u, v)=(\lambda \un x, \mu \un y, \lambda^{-1}\mu u, \lambda\mu^{-1} v).
 \]
 The subscheme $\calY$ of $\calX^0$ defined by the equation
 \begin{equation}\label{Qeq}
 u\cdot \lan \un x, \un x\ran_1+v\cdot \lan \un y, \un y\ran_2=0, 
 \end{equation}
 is stable under the torus action. The quotient
 \[
\calY/(\bbG_m\times\bbG_m)
 \]
 is isomorphic to $\cQ^\sharp(\Lambda, \Lambda^\vee)$ and is flat over $\O$. This shows 
$\cQ^\sharp(\Lambda, \Lambda^\vee)=\cQ(\Lambda, \Lambda^\vee)$ which is part (A).

Let us now discuss the subschemes $Z_0$, $Z_1$, $Z_2$, of the special fiber $\cQ^\sharp(\Lambda, \Lambda^\vee)_k=\cQ(\Lambda, \Lambda^\vee)_k$. Suppose that $R$ is a $k$-algebra and assume
 that $(L, L')$ is a pair of $R$-lines $L\subset \La_R$, $L'\subset \La^\vee_R$,
 which are linked, i.e. satisfy
 $i_R(L)\subset L'$ and $j_R(\pi L')\subset L$.
 Then $L'\subset   (\Lambda /\pi\Lambda^\vee)_R$ if and only if $j_R(\pi L')=(0)$. Similarly, $\pi^{-1}L\subset   (\Lambda^\vee / \Lambda )_R$ if and only if $i_R(L)=(0)$. 
We first observe that any pair of $R$-lines $(L, L')$ with $L'\subset   (\Lambda /\pi\Lambda^\vee)_R$, $\pi^{-1}L\subset   (\Lambda^\vee / \Lambda )_R$, satisfies $i_R(L)=(0)$, $j_R(\pi L')=(0)$ and so is linked, but also has $\lan L   , L'  \ran_{R}=0$ and hence it gives an $R$-point of $Z_0$. 
Suppose we are given an $R$-line $L'\subset   (\Lambda /\pi\Lambda^\vee)_R$ which is isotropic for the form $\lan\    ,\,   \ran_{1, R}$. Then every $R$-line $L\subset \Lambda_R$ with $i_R(L)= L'$ gives an $R$-point $(L, L')$ of $Z_1$. There is a similar construction for $R$-points of $Z_2$. Finally, observe that if $R$ is a $k$-algebra which is an integral domain, since $i_R\circ j_R=0$, at least one of $i_R(L)$, $j_R(\pi L')$, has to be $(0)$. If $i_R(L)=L'$, then $j_R(\pi L')=(0)$ and then $\lan L   , L'  \ran_{R}=0$ is equivalent to $\lan L'   , L'  \ran_{1, R}=0$. Similarly, if $j_R(\pi L')=L$, then $i_R(L)=(0)$ and $\lan L   , L'  \ran_{R}=0$ is equivalent to $\lan \pi^{-1}L   , \pi^{-1}L  \ran_{2, R}=0$. These considerations easily imply that the reduced special fiber of $\cQ(\Lambda, \Lambda^\vee)_k$ is the union $Z_0\cup Z_1\cup Z_2$. We can also deduce that the schemes  $Z_i$ have the descriptions given in (C1) and (C2). This will also be explained below.

To obtain more precise (scheme) theoretic information and show the remaining statements we will use the quotient description above. In this description, the subschemes $Z_2$, $Z_1$, $Z_0$  of $\cQ(\Lambda, \Lambda^\vee)$ are given by the quotient of the closures of $u=0$, $v\neq 0$, $\lan \un y, \un y\ran_2=0$, of $v=0$, $u\neq 0$, $\lan \un x, \un x\ran_1=0$, and of $u=v=0$, respectively:
 
 \begin{itemize}
 \item $Z_2$ is given by the quotient of the subscheme $\lan \un y, \un y\ran_2=0$ in the complement
 \[
 [\bbA(\bar M)\times_k (\bbA(\bar N)-\{0\})\times_k \bbA^1]-(0\times (\bbA(\bar N)-\{0\})\times 0)
 \]
 by $\bbG_m\times \bbG_m$. Here, $(\lambda, \mu)$ acts by  
 \[
 (\lambda, \mu)\cdot (\un x,  \un y, v )=(\lambda \un x,   \mu \un y, \lambda\mu^{-1} v).
 \]
 This is a projective space bundle over the $\bbG_m$-quotient of $\lan \un y, \un y\ran_2=0$ in $\bbA(\bar N)-\{0\}$;
 this last quotient is the projective quadric hypersurface $\cQ(\Lambda^\vee/\La, \lan , \ran_2)$. Hence, $Z_2$ is a projective space bundle over $\cQ(\Lambda^\vee/\La, \lan , \ran_2)$.
 
 \item Similarly, $Z_1$ is a projective space bundle over the projective quadric hypersurface\\ $\cQ(\La/\pi\La^\vee, \lan\, , \ran_1)$.
 
 \item $Z_0$  is the quotient 
\[
(\bbA(\bar M)-\{0\})\times_k (\bbA(\bar N)-\{0\})/(\bbG_m\times \bbG_m),
\] 
so it is a product of projective spaces $\bbP(\Lambda^\vee/\Lambda)\times \bbP(\Lambda/\pi\Lambda^\vee)$.
 \end{itemize}
  
 It  follows easily from the above that (C1) and (C2) hold.
 We now show the remaining statements, in particular that $\cQ(\La,\La^\vee)$ is a normal relative local complete intersection as claimed in (B), and that the identity of Weil divisors in (C3) is true. 

Consider a $\bar k$-point of $\cQ(\La,\La^\vee)$ obtained from a $\bar k$-valued point of $\calY$ for which $\und x=0$. Since $\calY\subset \calX^0$, we necessarily have $v\neq 0$
and hence $u=0$. Then we also have $\un y\neq 0$ and the point lies in $Z_2$ but not on $Z_0$. The equation (\ref{Qeq}) amounts to just $\lan \un y, \un y\ran_2=0$.
Since we are considering a non-zero point $\un y$, the scheme $\cQ(\La, \La^\vee)_k$ is smooth there.
The same argument works for a point for which $\und y=0$; this lies on $Z_1$ but not on $Z_0$.

It remains to deal with points for which $\un x\neq 0$ {\sl and} $\un y\neq 0$. Consider the corresponding open subscheme 
\[
\cQ^0(\La, \La^\vee)=[(\bbA(M)-\{0\})\times_\O (\bbA(N)-\{0\})\times \Spec(\O[u, v]/(uv-\pi))]/(\Gm\times\Gm)
\]
of $\cQ(\La, \La^\vee)$. We will obtain an explicit description of $\cQ^0(\La,\La^\vee)$:

Recall, we can choose bases of $M$ and $N$ that give coordinates  $\un x=(x_1,\ldots , x_{d-\delta})$ on $\bbA(M)$ and 
$\un y=(y_1, \ldots , y_\delta)$ on $\bbA(N)$. Denote by $Q_1(x_1, \ldots , x_{d-\delta})$,
$Q_2(y_1, \ldots , y_\delta)$ the  quadratic forms given as $\lan \un x, \un x\ran_1/2$
and $\lan \un y, \un y\ran_2/2$.

Then $\cQ^0(\Lambda, \Lambda^\vee)$ is covered
by the open affines $\calV_{i,j}$ with equations
\[
uv=\pi,\quad uQ_1(x_1, \ldots , x_{d-\delta})+vQ_2(y_1, \ldots , y_\delta)=0,\quad x_i=1,\quad y_j=1.
\]
For  simplicity, we will just consider the case $i=1$, $j=1$, and set $\calV=\calV_{11}$. 

We first consider the case $\delta=1$. Then $\calV$ is given by
\[
uv=\pi,\quad uQ_1(1, x_2, \ldots , x_{d-\delta})+v=0,
\]
i.e.
\[
-u^2 Q_1(1, x_2, \ldots , x_{d-\delta})=\pi.
\]
In this case, $\calV$ and hence $\cQ(\La,\La^\vee)$ is regular and its special fiber is a divisor with normal crossings. 

Assume now that $\delta\geq 2$. As also later, it helps to consider the simpler ``basic" scheme
\[
\calB=\Spec(\O[u, v, S, T]/(uv-\pi, uS+vT)
\]
There is a morphism $f: \calV\to\calB$ given by $S\mapsto Q_1$, $T\mapsto Q_2$.
We can see that $f$ is smooth. (Recall that $Q_1$ and $Q_2$ are  non-degenerate.) We will now check the desired properties for $\calB$: The scheme $\calB$ has relative dimension $2$ over $\Spec(\O)$ and is a relative complete intersection. The radical of $(uS+vT, uv)$ is $(uS, vT, uv)$ and so the reduced special fiber of $\calB$ has $3$ smooth irreducible components given by the prime ideals $(u, v)$,
$(S, v)$, $(T, u)$.  The component given by $(u, v)$ is not reduced in the special fiber;
the corresponding primary ideal is $(u^2, uv, v^2, uS+vT)$. The other components are reduced. 
We can see that $\calB$ is regular in codimension $1$ and it easily follows by Serre's criterion 
that $\calB$ is normal.

It now follows that we have a similar picture
for the special fiber of $\calV$ and that $\calV$ is a normal relative complete intersection of relative dimension $d-2$.
The result for $\cQ^0(\La,\La^\vee)$ and then also $\cQ(\La, \La^\vee)$ follows.
\end{proof}
\quash{Note that the inverse image of $(T,  u)$ in $\calV$
is 
\[
\Spec(k[v, x_2,\ldots , x_{d-\delta}, y_2,\ldots, y_{\delta}, ]/Q_2(1, y_2, \ldots , y_\delta)).
\]
which is a $\bbA^{d-\delta}_k$-bundle over the open  affine patch $Q_2(1, y_2, \ldots , y_\delta)=0$ of $Q_2=0$.
}

\begin{Remark}\label{delta1} {\rm Note that in the proof above, the case $\delta=1$ is special. 
\quash{Indeed, then $\calV$ is given by
\[
uv=\pi,\quad uQ_1(1, x_2, \ldots , x_{d-\delta})+v=0,
\]
i.e.
\[
-u^2 Q_1(1, x_2, \ldots , x_{d-\delta})=\pi.
\]
}
In this case, $\calV$ and hence $\cQ(\La,\La^\vee)$ is regular and its special fiber is a divisor with normal crossings. 
This result appears in \cite[12.7.2]{HPR}: There, $\cQ(\La,\La^\vee)$ is denoted as $P(\La)^{\rm fl}$ and is identified
with the blow-up of the singular quadric $\cQ(\La)$ of isotropic lines in $\La_R$ at the unique singular point of its special fiber $\cQ(\La)_k$
(see \emph{loc. cit.} Lemma 12.8). }
\end{Remark}

\subsection{A Blow-up} Now let us consider the blow-up 
\[
\beta_i: \wti \cQ^{(i)}(\La, \La^\vee)\to \cQ(\Lambda, \Lambda^\vee)
\]
 along the (reduced) subscheme $Z_i$, $i=1$, $2$.
 We first observe that when $\delta=1$, by Remark \ref{delta1}, $Z_1$ is locally principal and 
 so $\wti \cQ^{(1)}(\La, \La^\vee)= \cQ(\Lambda, \Lambda^\vee)$. 
 
 \begin{prop}\label{prop331}  
 
 a) There is a canonical isomorphism $\wti \cQ^{(1)}(\La, \La^\vee)\cong \wti \cQ^{(2)}(\La, \La^\vee)$. By abusing notation, we will denote both these schemes by $\wti \cQ(\La, \La^\vee)$.
 
 b) The scheme $\wti \cQ(\La, \La^\vee)$ is regular. Its special fiber is a divisor with (non-reduced) normal crossings and the multiplicity of each component is one or two. In fact, $\wti \cQ(\La, \La^\vee)$ is covered by open affine subschemes which are smooth over $
\Spec(\O[u, x, y]/(u^2xy-p))
$
when $\delta\geq 2$, or over
$
\Spec(\O[u, x]/(u^2x-p))
$
when $\delta= 1$. 

c) The reduced special fiber of $\wti\cQ(\La,\La^\vee)$ is a union of smooth irreducible 
 components. It has three irreducible components if $\delta>2$, four or three if $\delta=2$, and two if $\delta=1$.
 \end{prop}

\begin{proof} By Remark \ref{delta1} and the above, these statements hold when  $\delta=1$. Assume $\delta\geq 2$.
Let us consider the blow-up $\wti\calB$ of 
\[
\calB=\Spec(\O[u, v, S, T]/(uv-\pi, uS+vT)
\]
 along $(S, v)$. It is a closed subscheme of
\[
\{Sy-vx=0\ |\ (S, T, u, v)\times (x;y)\}\subset \calB\times_\O \bbP^1_\O,
\]
where $\bbP^1_\O$ has 
projective coordinates $(x;y)$. This is covered by two open affine subschemes:

I) $y=1$. Then $S=vx$ and we obtain
$
uv=\pi$, $v(ux+T)=0
$. Hence, we can see that the corresponding open of the blow-up $\wti\calB$
is cut out by the equations $uv=\pi$, $T=-ux$, in the subscheme displayed above.
This open is isomorphic to 
\[
\Spec(\O[u,v,x]/(uv-\pi)).
\]

II) $x=1$. Then $v=Sy$ and we obtain:
$
uSy=\pi$, $uS+SyT=0
$.
Hence,  the corresponding open of the blow-up $\wti\calB$
is cut out by the equations $uSy=\pi$, $u+yT=0$, in the subscheme displayed above.
These give $STy^2=-\pi$ and
this open is isomorphic to 
\[
\Spec(\O[S,T,y]/(STy^2+\pi)).
\]
Notice that the ideal $(T, u)$ becomes principal on both of these charts. On (I) we have
$T=-ux$, so $(T, u)=(u)$. On (II) we have $u=-yT$, so $(T, u)=(T)$. Using the universal property of the blow-up and symmetry gives that
the blow-ups of $(S, v)$ and $(T, u)$ are isomorphic.

By the proof of Theorem \ref{thmQ}, we see that the open subscheme $\cQ^0(\La, \La^\vee)\subset \cQ(\La, \La^\vee)$
is covered by open affine subschemes $\calV=\calV_{ij}$ that supports a smooth morphism
\[
f: \calV\to \calB=\Spec(\O[u, v, S, T]/(uv-\pi, uS+vT);\quad S\mapsto Q_1,\ T\mapsto Q_2.
\]
Note that, by the same proof, $\cQ(\La,\La^\vee)$ is smooth at the points of $\cQ(\La,\La^\vee)-\cQ^0(\La, \La^\vee)$ and 
the divisors $Z_1$ and $Z_2$ are principal at these points. Therefore, the blow-ups $q_i:\wti \cQ^{(i)}(\La, \La^\vee)\to \cQ(\La, \La^\vee)$ are 
isomorphisms locally over $\cQ(\La,\La^\vee)-\cQ^0(\La, \La^\vee)$. To show our result is enough to consider the blow-ups
of $\calV$ at $Z_1\cap \calV$ and at $Z_2\cap \calV$.
Since $Z_1\cap \calV$ is given by the ideal $(Q_1, u)$, the blow-up $\wti\calV_1$ of $\calV$ along $Z_1\cap \calV$ 
is obtained as the fiber product 
\[
\wti\calV_1\simeq \calV\times_{\calB }\wti\calB.
\]
The morphism $\wti\calV_1\to\wti\calB$ is smooth and (b) and (c) now follows from our discussion above and Theorem \ref{thmQ} and its proof.
Part (a), i.e. $\wti \cQ^{(1)}(\La, \La^\vee)\simeq \wti \cQ^{(2)}(\La, \La^\vee)$, follows from $\wti\calV_1\simeq \wti\calV_2$ since, by the observation above, the 
blow-ups of $\calB$ at $(S, v)$ and $(T, u)$ are isomorphic. 
\end{proof}

\begin{Remark}
{\rm It would be useful to have a simple moduli-theoretic description of the blow-up $\wti \cQ(\La, \La^\vee)$
similar in spirit to the description of $\cQ(\Lambda,\Lambda^\vee)$ given by Theorem \ref{thmQ} (A).

}
\end{Remark}

\section{Local models and variants}

\subsection{Local models after \cite{PZ}}
We  briefly recall some of the constructions in \cite{PZ}.

 Let $G$ be a connected reductive group over $F$. Assume that $G$ splits over a tamely ramified extension of $F$. Let $ \{ \mu \}$ be the conjugacy class of a minuscule geometric cocharacter $ \mu : \mathbb{G}_{m \bar F} \rightarrow G_{\bar F}$. Let $K$ be a parahoric subgroup of $G(F)$, which is the connected stabilizer of some point $x$ in the (extended) Bruhat-Tits building $\mathcal{B}^e(G,F)$ of $G(F)$.  Define $E$ to be the extension of $F$ which is the field of definition of the conjugacy class $ \{ \mu \}$.
 The construction of the local model $\Mloc_K(G,\{\mu\})$ is done as follows:

First, give an affine group scheme $\und\calG$ which is smooth over $\Spec(\O_F[t])$ and which, among other properties, satisfies:

\begin{enumerate}
    \item The base change of $\und\calG$ by $\Spec(\O_F) \rightarrow \Spec(\O_F[t]) = \mathbb{A}^1_{\O_F}$ given by $t \rightarrow \pi$ is the Bruhat-Tits group scheme $\calG$ which corresponds to  $K$ (see \cite{BTII}).  
    
    \item The group scheme $\und\calG\otimes_{\O_F [t]} \O_F [t, t^{-1}]$ is reductive.
\end{enumerate}

Next, consider the global (``Beilinson-Drinfeld") affine Grassmannian  
\[
{\rm Aff}_{\und\calG} \rightarrow \mathbb{A}^1_{\O_F}
\]
 given by $\und\calG$, which is an ind-projective ind-scheme. The base change $t \rightarrow \pi$, gives an equivariant isomorphism 
$$
 \text{Aff}_G \xrightarrow{\sim} \text{Aff}_{\und\calG }\times_{\mathbb{A}^1_{\O_F}} \Spec(F)
 $$ 
where $\text{Aff}_G$ is the affine Grassmannian of $G$; this is the ind-projective ind-scheme over $\Spec(F)$ that represents the fpqc sheaf associated to 
$$
R \rightarrow G(R\llps t\lrps)/G(R\lps t\rps),
$$
where $R$ is an $F$-algebra (see also \cite{PR}). The cocharacter $\mu$  gives an $\bar F\llps t\lrps$-valued point $ \mu(t)$ of $G$. This gives a $\bar F$-point $ [\mu(t)]= \mu(t) G ( \bar F\lps t\rps)$ of $\text{Aff}_G $.
Since $\mu$ is minuscule and $ \{\mu \} $ is defined over the reflex field $E$ the orbit 
$$
 G(\bar F\lps t\rps)[\mu(t)] \subset \text{Aff}_G (\bar F),
$$
is equal to the set of $ \bar F$-points of a closed subvariety $X_{\mu}$ of $\text{Aff}_{G,E}= \text{Aff}_G \otimes_{F }E .$

\begin{Definition} {\rm
The local model $\Mloc_K(G,\{\mu\})$ is the flat projective scheme over $\Spec(\O_E)$ with $\calG\otimes_{\O_F}\O_E$-action given by the reduced Zariski closure of the image of 
\[
X_{\mu} \subset \text{Aff}_{G, E} \xrightarrow{\sim} \text{Aff}_{\und\calG }\times_{\mathbb{A}^1_{\O_F}} \Spec(E)
\]
 in the ind-scheme $\text{Aff}_{\und\calG}\times_{\mathbb{A}^1_{\O_F}} \Spec(\O_E)$.}
\end{Definition}

These ``PZ" local models of \cite{PZ} are independent of choices in their construction (\cite[Theorem 2.7]{HPR}) and have the following property (see \cite[Prop. 2.14]{HPR} and its proof).
\begin{prop}\label{Unr}
 If $F'/F$ is a finite unramified extension, then (equivariantly)
 $$
 \Mloc_K(G, \{\mu\}) \otimes_{\O_E}\O_{E'}\xrightarrow{\sim } \Mloc_{K'}(G\otimes_FF',\{\mu\otimes_FF'\}). $$
Note that here the reflex field $E'$ of $(G\otimes_FF', \{\mu\otimes_FF'\})$ is the join of $E$ and $F'$. Also, $K'$ is the parahoric subgroup of $G\otimes_FF'$ with $ K = K' \cap G $.\qed
\end{prop}

\begin{Remark}\label{LMremark}
{\rm Let us note that the PZ local models are not well behaved when $p$ divides the order of the algebraic fundamental group $\pi_1(G_{\rm der})$.
To correct this defect, one needs to adjust the definition by using a $z$-extension of $G$, as in \cite[2.6]{HPR}.
The resulting variant is better behaved: It satisfies a property of invariance under central extensions and should 
agree with the local model conjectured to exist in \cite[21.4]{Schber}, see \cite[Conjecture 2.16]{HPR}. In this paper, 
we only consider   groups $G$ with adjoint group $G_{\rm ad}\simeq (\SO(V))_{\rm ad}$.
In this case, $|\pi_1(G_{\rm der})|$ is a power of $2$. Since we assume that $p$ is odd, the local models of \cite{HPR} coincide with the PZ local models given as above. In particular, the central extension invariance property of \cite[Prop. 2.14]{HPR} holds
in the cases we consider.}
\end{Remark}

\subsection{Lattices over $\mathcal{O}[u]$ and orthogonal local models}\label{ss43}

We now concentrate our attention to $G=\SO(V)$, where $V$ is an $F$-vector space of dimension $d\geq 5$ equipped with a non-degenerate symmetric $F$-bilinear form $\langle \; , \; \rangle$. We consider the minuscule coweight $ \mu : \mathbb{G}_m \rightarrow \SO(V) $ to be given by $ \mu(t) = \text{diag} ( t^{-1}, 1,\ldots ,1, t ) $, defined over $E=F$. We take $K$ to be the parahoric subgroup of $\SO(V)$ which is the connected stabilizer of a vertex lattice $\Lambda\subset V$, with $\delta(\La)={\rm length}_{\O_F}(\Lambda^\vee/\Lambda)$.

For simplicity, we set $\O=\O_F$. 
 We extend the data of the vector space $V$ with its symmetric bilinear form $\lan\ ,\ \ran$ from $F$ to $\mathcal{O}[u,u^{-1}]$ by following the procedure in \cite[5.2, 5.3]{PZ}. This is simpler to explain when we are in one of the $4$ cases of \S \ref{ss42a}
 (which we can assume after an unramified extension). Then, we define $\mathbb{V} = \oplus^{d}_{i=1} \mathcal{O}[u,u^{-1}] \und{e}_i $ and let $\langle \; , \; \rangle \; : \mathbb{V} \times \mathbb{V} \rightarrow  \mathcal{O}[u,u^{-1}] $ to be a symmetric $\mathcal{O}[u,u^{-1}]$-bilinear form such that the value of $\langle \und{e}_i , \und{e}_j \rangle$ is the same as the one for $V$ but with $\pi$ replaced by $u$. Similarly, we define $\und{\mu}  :  \mathbb{G}_m \rightarrow \SO(\mathbb{V}) $ as above by using the $\{ \un{e}_i\}$ basis for $\mathbb{V}$.
We set 
\[
\mathbb{L} = \oplus_{i=1}^d \mathcal{O}[u]\cdot \un{e}_i  \subset \mathbb V.
\]
 From the above, we see that the base change of $(\mathbb{V}, \mathbb{L},\langle \; , \; \rangle )$ from $\mathcal{O}[u,u^{-1}]$ to $F$ given by $u \mapsto \pi$ is $(V,\Lambda,\langle \; , \; \rangle)$.

Let us now consider the local model $\Mloc(\Lambda)= \Mloc_K (\SO(V),\{\mu\})$ where $K$ is the parahoric stabilizer $K_\La^\circ$ of $\Lambda$. 

As in \cite{PZ}, we consider the smooth, affine group scheme $\underline{\mathcal{G}}$ over $\mathcal{O}[u]$ given by $g \in \SO(\mathbb{V}) $ that also preserve $\mathbb{L}$ and $\mathbb{L}^{\vee}$. The base change of $\underline{\mathcal{G}}$ by $u \mapsto \pi $ gives the Bruhat-Tits group scheme $\mathcal{G}=\calG_\La$ of $\SO(V) $ which is the stabilizer of the lattice chain $\Lambda \subset \Lambda^{\vee} \subset \pi^{-1} \Lambda$ as in \S \ref{ss22}. 
In this case, the global (“Beilinson-Drinfeld”) affine Grassmannian $\text{Aff}_{\underline{\mathcal{G}}} \rightarrow \Spec(\mathcal{O}[u])$ represents the functor that sends an $\mathcal{O}[u]$-algebra $R$, with  $u$ mapping to $r$, to the set of finitely generated projective $R[u]$-submodules $\mathcal{L}$ of 
\[
\mathbb{V}\otimes_{\mathcal{O}} R[(u-r)^{-1}]=\oplus_i R[u, u^{-1}, (u-r)^{-1}] \un e_i
\]
 which are such that 
\[
(u-r)^{N}\mathbb{L }_R\subset \mathcal{L} \subset (u-r)^{-N}\mathbb{L}_R
\]
 for some $N >0$ with $\calL/(u-r)^N\bbL_R$, $(u-r)^{-N}\mathbb{L}_R/\calL$ both $R$-projective,
  and which satisfy 
\[
  \mathcal{L}\buildrel{\delta}\over\subset \mathcal{L}^{\vee}\buildrel{d-\delta}\over\subset u^{-1}\calL 
 \]
  with all successive quotients $R$-finite projective of the indicated rank.
  Here, we set $\bbL_R=\bbL\otimes_\O R$.

As in \cite[12.7.2]{HPR}, or \cite[3.b]{Zachos}, we see that the local model\footnote{Occasionally, if the choice of $\La$ is clear from the context, we will omit it from the notation.} $\Mloc=\Mloc(\La)$ is a subfunctor of ${\rm {Aff}}_{\un \calG}\otimes_{\O[u]}\O$
(with the base change given by $u\mapsto \pi$). 

\begin{Definition}\label{MnaiveDef}
{\rm We set ${\rm M}^{\rm naive}={\rm M}^{\rm naive}(\La)$ to be the functor (a subfunctor of ${\rm {Aff}}_{\un \calG}\otimes_{\O[u]}\O$),
which sends an $\O$-algebra $R$ to the set of $R[u]$-modules $\calL$ that satisfy:
\begin{enumerate}
    \item $\mathcal{L}$ is a finitely generated projective $R[u]$-submodule of $\bbL\otimes_{O[u]}R[u, u^{-1}, (u-\pi)^{-1}]$.

    \item $ \mathcal{L}\buildrel{\delta}\over\subset \mathcal{L}^{\vee}\buildrel{d-\delta}\over\subset u^{-1}\calL $, with all successive quotients $R$-finite projective and of the indicated rank,

    \item $
   (u-\pi)\mathbb{L }_R\subset \mathcal{L} \subset (u-\pi)^{-1}\mathbb{L}_R$, with the quotients  $\calL/(u-\pi)\bbL_R$, $(u-\pi)^{-1}\mathbb{L}_R/\calL$, both $R$-finite projective of rank $d$,
   \smallskip
   
   \item
   $
   (u-\pi)\bbL^\vee_R \subset    \mathcal{L}^\vee \subset   (u-\pi)^{-1}\bbL^\vee_R$, with  
   $\mathcal{L}^\vee/ (u-\pi)\bbL^\vee_R$, $(u-\pi)^{-1}\mathbb{L}^\vee_R/\calL^\vee$, both $R$-finite  projective of rank $d$,
\smallskip

\item Consider the $R$-linear map 
\[
\Phi: \calL/(u-\pi)\bbL_R\to (u-\pi)^{-1}\bbL_R/\bbL_R= \Lambda_R
\]
given by the inclusion in (3). This is an $R$-map between two finitely generated projective $R$-modules of rank $d$.
We require that $\Phi$ has $R$-rank $\leq 1$, i.e. that 
\[
\wedge^2 \Phi=0.
\]

\item
Similarly, we consider
\[
\Psi: \calL^\vee/(u-\pi)\bbL^\vee_R  \to (u-\pi)^{-1}\bbL^\vee_R/\bbL^\vee_R=\La^\vee_R,
\]
given by the inclusion in (4), and require that 
\[
\wedge^2 \Psi=0.
\]
\end{enumerate}}
\end{Definition}
 
\begin{prop}\label{Mnaive} The functor ${\rm M}^{\rm naive}(\La)$ is represented by a closed projective subscheme $ {\rm M}^{\rm naive}(\La)$ of ${\rm {Aff}}_{\un \calG}\otimes_{\O[u]}\O$ with $\calG$-action. There is a $\calG$-equivariant closed immersion
\[
i: \Mloc(\La)\hookrightarrow {\rm M}^{\rm naive}(\La).  
\] 
\end{prop}

\begin{proof}
The proof of the first part (representability) is standard. The second part follows from the construction of $\Mloc(\La)$.
See \cite[12.7.2]{HPR}, or \cite[3.b]{Zachos}, for more details.
\end{proof}

\begin{Remark} 
{\rm The immersion $i$ is not an isomorphism, i.e.
the conditions (1)-(6) are necessary but not always sufficient for 
$\calL$ to correspond to an $R$-valued point of $\Mloc(\La)$. 

In fact,  the generic fibers ${\rm M}^{\rm naive}(\La)\otimes_\O F$
and $\Mloc(\La)\otimes_\O F$ are {\sl not} equal. Indeed, ${\rm M}^{\rm naive}(\La)(F)$ contains the additional $F$-point $\calL=\bbL_F$ which is not in the orbit $X_\mu=\cQ(V)$ of $\und\mu$ in ${\rm {Aff}}_{\un \calG}\otimes_{\O[u]}F={\rm {Aff}}_G$. We can easily see that the reduced locus of ${\rm M}^{\rm naive}(\La)\otimes_\O F$ decomposes into the (disjoint) union of $\cQ(V)=\Mloc(\La)\otimes_\O F$ with this point.
By its definition,
$\Mloc(\La)$ is the (reduced) Zariski  closure of $\cQ(V)$ in ${\rm M}^{\rm naive}(\La)$.

However,  ${\rm M}^{\rm naive}(\La)$ is not very different from $\Mloc(\La)$: We will see in \S\ref{coro} that 
${\rm M}^{\rm naive}(\La)$ is the push-out
of $\Mloc(\La)$ and $\Spec(\O)$, ``glued" at the point $*=\Spec(k)$. In particular,  when we just regard
 the underlying topological spaces, the image of $i$ only misses the isolated point  of the generic fiber of ${\rm M}^{\rm naive}(\La)$ that
 corresponds to $\calL=\bbL_F$. The scheme ${\rm M}^{\rm naive}(\La)$ is $\O$-flat but has non-reduced special fiber with the non-reduced locus supported at $*$.
  }
  \end{Remark}

  \subsection{Spinor local models}\label{ssSpinLM}
 Recall the constructions and considerations of \S \ref{ss23} and \S \ref{ss24}. 
 In particular, set $G=\GSpin(V)$ and recall the central extension
 \[
 1\to \Gm\to \GSpin(V)\xrightarrow{\alpha} \SO(V)\to 1.
 \]
 Define the cocharacter $\wti \mu:\Gm\to G$ by
\[
\mu(t)=t^{-1}f_1f_{d}+f_{d}f_1 
\]
where $f_1$, $f_d$ are part of a basis $(f_i)_i$ of $V$ such that $\lan f_1, f_d \ran=1$, $\lan f_1, f_1\ran=\lan f_d, f_d\ran=0$,
and the arithmetic on the right hand side takes place in the Clifford algebra $C(V)$.
 Under the  representation $\alpha: G \to \SO(V)$, we have 
\[
(\alpha\cdot \wti\mu(t)) \cdot f_i = \wti\mu(t) f_i\wti\mu(t)^{-1}=\begin{cases}
t^{-1} f_i & \mbox{if }i=1 \\
f_i & \mbox{if } 2 \le i < d \\
t f_i & \mbox{if }i=d. 
\end{cases}
\]
Hence, $\alpha\cdot \wti\mu: \Gm\to \SO(V)$ is given by the minuscule $\mu(t)={\rm diag}(t^{-1}, 1, \ldots , 1, t)$.

Now let $x\in \calB^e(G, F)$ be a point in the extended Bruhat-Tits building of $G(F)$ such that $x$ maps to $x_\La$
under $\alpha_*: \calB^e(G, F)\to \calB(\SO(V), F)$, where $\La$ is a vertex lattice of $V$. Then the corresponding parahoric group of 
$G$ is given by
\[
K=\calG_x(\O)=\{g\in \GSpin(V)(F)\ |\ g\Lambda g^{-1}= \Lambda, \ \eta(g)\in \O^\times\}.
\]
By Remark \ref{LMremark} and \cite[Proposition 12.4]{HPR}, there is an equivariant isomorphism between local models
\begin{equation}\label{spinlocm}
\Mloc_{ K}(G, \{\wti\mu\})\xrightarrow{\sim} \Mloc_{K^\circ_\La}(\SO(V), \{\mu\})=\Mloc(\La)
\end{equation}
where $K^\circ_\La$ is the parahoric stabilizer of $\La$ in $\SO(V)$. 

The equivariance here is meant in the following sense: 
The natural action of $\calG_x=\calG^\circ_x$ on $\Mloc_{K}(G, \{\wti\mu\})$
factors via the quotient $\calG_x/\Gm\simeq \calG^\circ_\La$ (see (\ref{ces3})) and, under the isomorphism 
(\ref{spinlocm}), agrees with the action of the corresponding parahoric group scheme $\calG^\circ_\La$ of $\SO(V)$
on $\Mloc(\La)$. 

Hence, all our results on $\Mloc(\La)$, are really also about
local models for the spin similitude groups.

\section{Equations and a resolution for the local model}\label{sEq}

 \subsection{An open affine of the local model following \cite{Zachos}}\label{ss44}

 We continue with the same notations, see especially \S \ref{ss43}. Our goal is to describe explicitly certain open affine subschemes $U^{\rm naive}$
 and $U$ of ${\rm M}^{\rm naive}(\La)$ and $\Mloc(\La)$, respectively. We will work over $\br \O$ but, for simplicity, sometimes omit the base
change from the notation.

We start with $U^{\rm naive}$. For an $R$-valued point $\calL$ of ${\rm M}^{\rm naive}={\rm M}^{\rm naive}(\La)$ we set
\[
\bar\calL:={\rm Im}(\calL)\subset {(u-\pi)^{-1}\bbL_R}/{(u-\pi)\bbL_R}
\]
which determines $\calL$ uniquely.\footnote{Under
$
(u-\pi)^{-1}\bbL/(u-\pi)\bbL\xrightarrow{\ \sim\ } { \bbL_R}/{(u-\pi)^2\bbL_R}
$
given by multiplication by $(u-\pi)$, the module $\bar\calL$ can  be identified with 
$
\bar\calL={\rm Im}((u-\pi): \calL\to  { \bbL}_R/{(u-\pi)^2\bbL_R})
$
used in \cite[3.b]{Zachos}.}

Recall we are fixing a basis 
\[
\bbL=\oplus_{i=1}^d\br\O[u]\cdot \und e_i
\]
which corresponds to the basis $(e_i)$ of $\Lambda=\bbL/(u-\pi)\bbL$. This, in turn, gives an $R$-basis
$\{\und e_i, (u-\pi)^{-1}\und e_i\}$ of $(u-\pi)^{-1}\bbL_R/\bbL_R$, for any $\O$-algebra $R$. We  consider the
open subscheme $U^{\rm naive}$ of ${\rm M}^{\rm naive}$ where $\bar\calL$ projects isomorphically onto the quotient obtained by 
annihilating the second part of the basis, i.e. $\{(u-\pi)^{-1}\und e_i\}$. We will now describe $U^{\rm naive}$ explicitly. 

We set
\[
\bar\calL=\{v +X(u-\pi)^{-1}  v\ |\ v=\sum^d_{i=1} v_i \un e_i\}
\]
where $X$ is a variable $d\times d$ matrix with entries in $R$.
This naturally produces an $R$-basis $\un e_i+X(u-\pi)^{-1}\un e_i$ for $\bar\calL$, i.e. an $R$-module isomorphism $R^d\simeq \bar\calL$.
Then the map 
\[
\Phi: \calL/(u-\pi)\bbL_R=\bar\calL\hookrightarrow (u-\pi)^{-1}\bbL_R/(u-\pi)\bbL_R\to (u-\pi)^{-1}\bbL_R/\bbL_R
\]
 of \S \ref{ss43} (5),  is given by the matrix $X$. 
 
Similarly, we can write
\[
\bar\calL^\vee=\{v +Y(u-\pi)^{-1} v\ |\ v=\sum^d_{i=1} v_i \un e^\vee_i\}.
\]
where $Y$ is a variable $d\times d$ matrix with entries in $R$.
As above, the map 
\[
\Psi: \calL^\vee/(u-\pi)\bbL^\vee_R = \bar\calL^\vee \hookrightarrow (u-\pi)^{-1}\bbL^\vee_R/(u-\pi)\bbL^\vee_R\to (u-\pi)^{-1}\bbL^\vee_R/\bbL^\vee_R
\]
 of \S \ref{ss43} (6), is given by the matrix $Y$. 
 
 a) The condition that $\calL^\vee $ is the dual of $\calL$, gives that
we have $\lan \bar\calL, \bar\calL^\vee\ran=0$ under the $R$-base change of the pairing
\[
\lan\ ,\ \ran: (u-\pi)^{-1}\bbL/(u-\pi)\bbL\times (u-\pi)^{-1}\bbL^\vee/(u-\pi)\bbL^\vee\to (u-\pi)^{-2} \br\O[u]/\br\O[u].
\]
Hence, on this affine chart 
\[
\lan v+X(u-\pi)^{-1} v, w+Y(u-\pi)^{-1} w\ran=0.
\]
This is equivalent to
\begin{equation}\label{c1}
Y+X^t=0,\qquad X^t\cdot Y=0
\end{equation}
since $\lan Xv, w\ran=\lan w, X^tw\ran$. 

b) The conditions $\wedge^2\Phi=0$, $\wedge^2\Psi=0$, immediately translate to
\begin{equation}\label{c2}
\wedge^2X=0,\qquad \wedge^2 Y=0.
\end{equation}

c) As in \cite[3.c]{Zachos}, we see that the condition $ \mathcal{L} \subset  \mathcal{L}^{\vee}$ amounts to
\begin{equation}\label{c3}
  X^{t} S_1 X - 2 \pi S X=0, \quad X^{t} S_2 X  + 2 SX =0.
\end{equation}
Similarly, the condition $\calL^\vee \subset u^{-1}\calL $ amounts to
\begin{equation}\label{c4}
  Y^{t} S_1 Y + 2 (S_2 Y + \pi S_1 Y)=0,\quad Y^{t} S_2 Y  - 2\pi (S_2 Y + \pi S_1 Y ) =0. 
\end{equation}
In the above, $S_1$, $S_2$, are the matrices with $S=S_1+\pi S_2=(\lan e_i, e_j\ran)_{i, j}$ as in \S \ref{ss42a}.

d) Finally, $\bar\calL$ and $\bar\calL^\vee$ should be $u$-stable. This translates to 
\[
X^2=0,\qquad Y^2=0.
\] 
However, these are implied by the equations (\ref{c1}) above.

Denote by $J_m$ the unit antidiagonal matrix of size $m$,
\[
 J_m := \begin{pmatrix} 
          &  & 1  \\
          & \iddots  & \\
         1 &  &   \\
  \end{pmatrix}. 
    \]

We now set $X=(x_{ij})$, $Y=(y_{ij})$, and denote by  $\mathfrak I^{\text{naive}}$ the ideal of the polynomial ring $\br\O[(x_{ij}), (y_{ij})]$ in $2d^2$ variables, which is generated by the entries of the above relations (\ref{c1}), (\ref{c2}), (\ref{c3}), (\ref{c4}). Also denote by $U^{\text{naive}}$ the corresponding scheme over $\Spec(\br\O)$ given by pairs of $d \times d$ matrices $X$, $Y$ which satisfy these relations (for more details see Section $3$ in \cite{Zachos}).  
Set 
\[
{U}^{\text{naive}}=\Spec(\br\O[X, Y]/\mathfrak I^{\rm naive})
\]
which is an open subscheme of ${\rm M}^{\rm naive}(\La)$. Proposition \ref{Mnaive} realizes $\Mloc(\La)$ as a closed subscheme of ${\rm M}^{\rm naive}(\La)$, and we can set
\[
U=U^{\rm naive}\cap \Mloc(\La)=\Spec(\br\O[X, Y]/\mathfrak I),
\]
with $\mathfrak I^{\rm naive}\subset \mathfrak I$. The scheme $U$ is an open affine subscheme of $\Mloc(\La)$.
Denote by $\br\O(U^{\rm naive})$, resp. $\br\O(U)$, the affine coordinate rings of $U^{\rm naive}$, resp. $U$, over $\br\O$.

Following \cite[3.c]{Zachos}, we distinguish two cases:

I) The integers $d$ and $\delta$ have the same parity (cases (1) and (4) of \S \ref{ss42a}):

We then break up the matrix $X$ into blocks as follows. We write
\begin{equation}\label{mat.X}
X =\left[\ 
\begin{matrix}[c|c|c]
D_1 & C_1 & D_2 \\ \hline
B_1 & A & B_2 \\ \hline
D_3 & C_2 & D_4
\end{matrix}\ \right],
\end{equation}
where $A$ is of size $\delta\times\delta$, $D_1$, $D_2$, $D_3$, $D_4$, are of size $(n-r)\times (n-r)$. (Recall, $d=2n$, $\delta=2r$, or
$d=2n+1$, $\delta=2r+1$.) We set 
\[
{\mathrm {T}}(B_1|B_2)={\rm Tr}(B_2 J_{n-r} B_1^{t}J_{\delta}).
\]

II)  The integers $d$ and $\delta$ have different parity (cases (2) and (3) of \S \ref{ss42a}). In this case, we again decompose
the matrix $X$ in $9$ blocks as above, but the recipe for the dimensions of these blocks is somewhat different.
In order to define the submatrices $A$, $B_{i}$, $C_j$, $D_j$,  giving the block decomposition of $X$ we set: 
$$
r' =
\left\{\begin{array}{ll} r  & \mbox{if } \delta =2r  \\ r+1 & \mbox{if } \delta=2r+1. \end{array} \right.
$$
Then we write the matrix $X$ as before, with blocks $ A$ of size $(\delta+1) \times (\delta+1)$, $D_1$, $D_2$, $D_3$, $D_4$ of size $(n-r') \times (n-r')$.

We denote by $A' $ the $\delta\times \delta$ matrix which is obtained from $A$ by erasing the part that is in the $(n+1)$-row and $(n+1)$-column of $X$. Similarly we denote by $B'_1,B'_2$ the $\delta \times (n- r')$ matrices which are obtained from $B_1$, $B_2$ by erasing the part that is on the $(n+1)$-row of $X$. Lastly, we denote by $E $ the $(r'+1 )$-column of $A$ and $E'$ the  $(r'+1 )$-column of $A$ with the $(n+1)$-entry erased. (Recall, $d=2n$, $\delta=2r+1$, or
$d=2n+1$, $\delta=2r$.)  We set
\[
\mathrm T(B'_1| E'| B'_2)= {\rm Tr}(  ( B'_2 J_{n-r'} (B'_1)^{t}  +\frac{1}{2}   E' (E')^{t})J_{\delta} ).
\]

Finally, for simplicity, we set
\[
Z=
\begin{cases}
[B_1|B_2], \ \ \ \ \ \hbox{\rm if\ }d\equiv\delta\, {\rm mod}\ 2,\\
[B_1'|E'|B'_2], \ \hbox{\rm if\ } d\not\equiv\delta\, {\rm mod}\ 2.
\end{cases}
\]
Then $Z=(z_{ij})\in {\rm Mat}_{\delta\times (d-\delta)}$, in both cases.

\begin{thm}\label{Za1}  
(\cite{Zachos}) The inclusion  $\br\O[Z]\to \br\O[X, Y]$ induces isomorphisms
\begin{eqnarray}
  {\br\O[Z]}/{( \wedge^2 Z,  {\mathrm {T}}(Z)+ 2 \pi)}&\xrightarrow{\sim}& \br\O(U) ,\\
 {\br\O[Z]}/{( \wedge^2 Z, ({\mathrm {T}}(Z) + 2 \pi)\cdot Z)} &\xrightarrow{\sim} &  \br\O(U^{\rm naive})  .
 \end{eqnarray}
 \end{thm}
 
\subsection{}\label{coro} 
This is essentially contained in \emph{loc. cit.} but, for completeness, we will also give the
argument below. Before we do that, we discuss some corollaries:
 
 A) By an explicit calculation, we find that
 \[
 \mathrm T(Z)=\frac{1}{2}\sum_{1\leq i\leq \delta, 1\leq j\leq d-\delta} z_{i\ d-\delta+1-j}\, z_{\delta+1-i\ j}.
 \]
 (The same expression is valid in both cases, of same (I), or different (II), parity.)
 The result stated in the introduction follows: 
 
 Denote by $\mathscr D^{2}_{\delta\times (d-\delta)}=\{Z\ |\ \wedge^2Z=0\}\subset {\rm Mat}_{\delta \times (d-\delta)}$ the ``determinantal" subscheme of the affine space of matrices $Z$ over $\Spec(\br\O)$. 

\begin{thm}\label{Thmexplicit}
The affine chart $U\subset \Mloc(\La)$ is isomorphic to the closed  subscheme $\mathscr D_{\rm T}$ of the determinantal scheme $\mathscr D^{2}_{\delta\times (d-\delta)}$ which is
defined by the quadratic equation
\[
\sum_{1\leq i\leq \delta, 1\leq j\leq d-\delta} z_{i\ d-\delta+1-j}\, z_{\delta+1-i\ j}=-4\pi.
\]
\end{thm}
\begin{Remark}
{\rm As in \cite[Section 5]{Zachos}, one can see, by using the classical result that the determinantal scheme is Cohen-Macaulay, that
$\mathscr D_{\rm T}$ above is $\br\O$-flat, Cohen-Macaulay and of relative dimension $d-2$. This observation is used in the proof 
of Theorem \ref{Za1}: It is needed to establish that the $\br\O$-algebra giving $U$ in the statement above is indeed  $\br\O$-flat. 
It easily follows that $\Mloc(\La)$ is also Cohen-Macaulay. }
\end{Remark}

B) By Theorem \ref{Za1},
\begin{equation}
\br\O(U)=\br\O(U^{\rm naive})/({\rm T}(Z)+2\pi),
\end{equation}
where we slightly abuse notation by denoting by ${\rm T}(Z)+2\pi$ the image of this element in the quotient ring $\br\O(U^{\rm naive})$.
We can easily check that the annihilator of the element ${\rm T}(Z)+2\pi$
 in the coordinate ring $\br\O(U^{\rm naive})$ is the ideal 
$(Z)$ generated by the entries of $Z$. 
We obtain an exact sequence
\[
0\to \br\O(U^{\rm naive})\to \br\O(U)\times \br\O\xrightarrow{\ } k\to 0,
\]
where the second map is the difference of the reductions modulo the maximal ideals $(Z)+ (\pi)$ and $(\pi)$.
Hence,
\[
\br\O(U^{\rm naive})=\br\O(U)\times_{k} \br\O
\]
where, on the right hand side, we have the fibered product of rings.
This exhibits $U^{\rm naive}$ as the push-out
\begin{equation*}\label{pushoutU}
\begin{tikzcd}
\Spec(\bar k)\arrow[r, hook]\arrow[d, hook]&\Spec(\O)\arrow[d, hook, "X=Y=0"]\\
U \ \arrow[r, hook, "i"] & \ {U}^{\rm naive}. 
\end{tikzcd}
\end{equation*}
It easily follows that ${\rm M}^{\rm naive}(\La)$ is also a push-out:
 \begin{equation}\label{pushout}
\begin{tikzcd}
\Spec(k)\arrow[r, hook]\arrow[d, hook, "*"']&\Spec(\O)\arrow[d, hook, "\calL=\bbL"]\\
\Mloc(\La) \ \arrow[r, hook, "i"] & \ {\rm M}^{\rm naive}(\La). 
\end{tikzcd}
\end{equation}

We now give the proof of Theorem \ref{Za1}.

\begin{proof} This is a variation of the work in \cite{Zachos}. 
We explain a simpler version of the argument. (We will also omit some of the 
explicit calculations.)

We start with the case (I) of same parity $d\equiv\delta\, {\rm mod}\, 2$.

First,  we prove that $\br\O(U^{\rm naive})$, resp. $\br\O(U)$, are quotient rings of $\br\O[B_1, B_2]$:

Let $\mathfrak R$ be the ideal generated by (the entries of) the elements $(1)$-$(8)$:
\begin{multicols}{2}
    \begin{enumerate}
         \item $D_1 +\frac{1}{2} J_{n-r} B_2^{t}J_{\delta} B_1,$
        \item $D_2 +\frac{1}{2} J_{n-r} B_2^{t}J_{\delta} B_2 ,$
        \item $ D_3 +\frac{1}{2} J_{n-r} B_1^{t}J_{\delta} B_1 ,$
        \item $D_4 +\frac{1}{2} J_{n-r} B_1^{t}J_{\delta} B_2 ,$
        \item $C_1 +\frac{1}{2} J_{n-r} B_2^{t}J_{\delta} A  ,$
        \item $C_2 +\frac{1}{2} J_{n-r} B_1^{t}J_{\delta} A $,
        \item $ A - B_2 J_{n-r} B_1^{t}J_{\delta}$,
        \item $Y+X^t$.
        \end{enumerate}
    \end{multicols}
    We observe that 
    \[
\mathfrak R \subset  \mathfrak I^{\rm naive}.
\]
Indeed, the first six relations (1)-(6) are implied from  the relation $X^{t} S_1 X - 2 \pi S X\in \mathfrak I^{\rm naive}$ (\ref{c3}), relation (7) from (\ref{c1}) and (\ref{c4}), and (8) from (\ref{c1}).

 Hence, the relations (1)-(8) express each block $A$, $D_i$, $C_j$ of $X$ and $Y$, in terms of $B_1$ and $B_2$, modulo 
$\mathfrak I^{\rm naive}$. It follows that $\br\O(U^{\rm naive})$ and, therefore, also $\br\O(U)$, is a quotient of $\br\O[B_1, B_2]$.
In fact, since $\mathfrak I^{\rm naive}$ contains all the $2\times 2$ minors of $X$, it follows that $\br\O(U^{\rm naive})$ and  $\br\O(U)$ are also quotients of $\br\O[B_1, B_2]/(\wedge^2[B_1|B_2])$. This last ring is the affine coordinate ring of the cone over the Segre embedding of $\bbP^{\delta-1}\times \bbP^{d-\delta-1}$ over $\br\O$ and so it is integral and flat of relative dimension $d-1$ over $\br\O$.

We now continue to uncover additional relations in $\mathfrak I^{\rm naive}$. 
As is observed in \cite[4.a]{Zachos}, the condition $\wedge^2 X=0$
easily gives
\begin{equation}\label{TrAB}
AB_1={\rm Tr}(A)B_1, \quad AB_2={\rm Tr}(A)B_2.
\end{equation}
By \cite[Lemma 4.4]{Zachos}, the relation $\wedge^2 X=0$ implies that $B_1^t J_{n-r} B_2$ is symmetric,
so, modulo $\mathfrak I^{\rm naive}$, we have
\[
B_1^t J_{n-r} B_2=B_2^t J_{n-r} B_1.
\]
By looking at the appropriate blocks, we see that the relations (\ref{c3})    imply
\[
A^t J_\delta B_1+2\pi J_\delta B_1\in \mathfrak I^{\rm naive}, \qquad A^t J_\delta B_2+2\pi J_\delta B_2 \in \mathfrak I^{\rm naive}.
\]
Now $A - B_2 J_{n-r} B_1^{t}J_{\delta} \in \mathfrak I^{\rm naive}$ (relation (7) above), 
so, modulo $\mathfrak I^{\rm naive}$,
\begin{eqnarray*}
A^t J_\delta B_1+2\pi J_\delta B_1&=&(B_2 J_{n-r} B_1^{t}J_{\delta})^t J_\delta B_1+2\pi J_\delta B_1\\
&=&J_\delta B_1 J_{n-r} B_2^t J_\delta B_1+2\pi J_\delta B_1\\
&=&J_\delta A B_1+2\pi J_\delta B_1\\
&=& J_\delta({\rm Tr}(A)+2\pi)B_1,
\end{eqnarray*}
where the last identity uses (\ref{TrAB}).
This implies that $({\rm Tr}(A)+2\pi)B_1\in \mathfrak I^{\rm naive}.$ Similarly, starting from 
\[
A^t J_\delta B_2+2\pi J_\delta B_2\in \mathfrak I^{\rm naive},
\]
we obtain $({\rm Tr}(A)+2\pi)B_2\in \mathfrak I^{\rm naive}.$ Since ${\rm T}(B_1|B_2)={\rm Tr}(A)$ modulo 
$\mathfrak I^{\rm naive}$, this gives that $\br\O(U^{\rm naive})$ is a quotient of
$\br\O[B_1, B_2]/(\wedge^2 [B_1|B_2], ({\rm T}(B_1|B_2)+2\pi)[B_1|B_2])$.

Observe that $B_1=B_2=0$ gives $X=Y=0$, which corresponds to $\calL=\bbL$. This point
does not belong to the generic fiber of the $\br\O$-flat $\Mloc$. Hence, we find that $\br\O(U)$ is a quotient
of $\calA=\br\O[B_1, B_2]/(\wedge^2 [B_1|B_2], {\rm T}(B_1|B_2)+2\pi)$. The ring $\calA$ 
is the coordinate ring of a hypersurface in the integral Cohen-Macaulay 
$\br\O[B_1, B_2]/(\wedge^2 [B_1|B_2])$;
we can easily see as in \cite[Section 5]{Zachos} that $\calA$ is integral of dimension $d-2$
and $\br\O$-flat. Since $U$ also shares both these properties and $\br\O(U)$ is a quotient 
of $\calA$, it follows that $\br\O(U)=\calA$, as we wanted. The result for $\br\O(U^{\rm naive})$ 
also quickly follows: Indeed, $\br\O(U^{\rm naive})$ is a quotient of 
\[
\calA_+:=\br\O[B_1, B_2]/(\wedge^2 [B_1|B_2], ({\rm T}(B_1|B_2)+2\pi)[B_1|B_2]).
\]
However, this quotient has to be large enough to also allow both algebras
$
\br\O(U)=\calA
$
and $\br\O=\O[B_1, B_2]/([B_1|B_2])$ to appear as quotients. The corresponding spectrum has to support a morphism from the
push-out of $U$ and $\Spec(\br\O)$, glued at the point $X=Y=\pi=0$. But, as in \S\ref{coro} above, this push-out is the spectrum of 
$\calA_+$,
so 
\[
\br\O(U^{\rm naive})=\calA_+=\br\O[B_1, B_2]/(\wedge^2 [B_1|B_2], ({\rm T}(B_1|B_2)+2\pi)[B_1|B_2]),
\]
as we wanted.

The case (II) of different parity is similar: The role of $A$ is now played by $A'$ and
the relation (7) above is replaced by
\[
A'-(B'_2 J_{n-r'} (B'_1)^{t}   +  \frac{1}{2} E'(E')^{t}) J_\delta \in \mathfrak I^{\rm naive}
\]
which gives ${\rm Tr}(A')={\rm T}(B'_1|E'|B_2')$ modulo $\mathfrak I^{\rm naive}$.
See also the proof of \cite[Theorem 3.3]{Zachos} for more details.
\end{proof}

 \subsection{The blow-up of the local model $\Mloc$}\label{ssBl}
 
Let 
\[
r^{\rm bl}: \Mbl(\La)\to \Mloc(\La)
\]
 be the blow-up of $\Mloc(\La)$ at the closed point $*$ of its special fiber that corresponds to $\calL=\bbL$.
 We will now show Theorem \ref{BUpIntro} of the introduction: 
 
 \begin{proof} By Theorem \ref{Thmexplicit}, it is enough to show 
 the conclusion of the theorem for the blow-up 
 $\wti {\mathscr D}_{\rm T}$ of ${\mathscr D}_{\rm T}$ at the (maximal) ideal given by $(z_{ij})$.
 For simplicity, we write $\mathscr D$ for the determinantal scheme ${\mathscr D}^2_{\delta\times (d-\delta)}$ over $\Spec(\br\O)$.
This is the affine cone over the Segre embedding
 \[
 (\bbP^{\delta-1}\times \bbP^{d-\delta-1})_{\br\O}\hookrightarrow \bbP^{\delta(d-\delta)-1}_{\br\O}.
 \]
 Also, we set
\[
{\rm T}=\frac{1}{2}\sum_{1\leq i\leq \delta, 1\leq j\leq d-\delta} z_{i\ d-\delta+1-j}\, z_{\delta+1-i\ j}.
\]
Let us consider the blow-up 
\[
\wti {\mathscr D} \tol {\mathscr D} 
\]
of the determinantal scheme over $\Spec(\br\O)$
along the vertex of the cone, i.e. along the subscheme defined by the ideal $(z_{ij})$. 
Then, the blow-up $\wti {\mathscr D}_{\rm T}$ is isomorphic to the strict transform of the hypersurface $\mathscr D_{\rm T}\subset \mathscr D$ given by
${\rm T}+2\pi=0$. By definition, $\wti {\mathscr D}$ is a closed subscheme of ${\mathscr D}\times_{\br\O} \bbP^{\delta(d-\delta)-1}_{\br\O}$ and in fact,
\[
\wti {\mathscr D}= \{(z_{ij}, u_{i,j})\ |\ (z_{ij}u_{i',j'}-z_{i'j'}u_{i,j}),\, \wedge^2(z_{ij})=0,\, \wedge^2(u_{i,j})=0\} \subset {\mathscr D}\times_{\br\O} \bbP^{\delta(d-\delta)-1}_{\br\O}
\]
where $(u_{i,j})$ are homogeneous coordinates on $\bbP^{\delta(d-\delta)-1}_{\br\O}$.
Let $V_{s, t}$ be the open affine  of $\wti {\mathscr D}$ over which the image of $z_{s t}$ generates the pull-back of the ideal $(z_{ij})$. 
Then 
\[
V_{s, t}\simeq \Spec(\br\O[(u_{i,j})_{1\leq i\leq \delta, 1\leq j\leq d-\delta}]/((u_{i,j}-u_{s, j} u_{i, t})_{i, j}, u_{s, t}-1).
\]
The intersection $V_{s, t}\cap \wti {\mathscr D}_{\rm T}$ is obtained by substituting $z_{ij}=u_{i,j}z_{st}$
and $u_{i,j}=u_{s, j}u_{i,t}$, for all $i$, $j$, in the equation ${\rm T}=-2\pi$. 
This amounts to setting
\[
z_{ij}=u_{s,j}u_{i,t}z_{st}
\]
and gives
\[
4\pi+z_{st}^2 (\sum_{1\leq i\leq \delta, 1\leq j\leq d-\delta } u_{s, d-\delta+1-j} u_{i, t} u_{s,j} u_{\delta+1-i, t})=0.
\]
This is
\begin{equation}\label{quadBU}
4\pi+z_{st}^2 (\sum_{i=1}^\delta u_{i,t}u_{\delta+1-i, t})(\sum_{j=1}^{d-\delta} u_{s,j} u_{s, d-\delta+1-j})=0.
\end{equation}
Note that, since $u_{s,t}=1$, the two sums in the line above are
\[
S_1=u_{\delta+1-s, t} +\sum_{i\neq s} u_{i,t}u_{\delta+1-i, t}\ ,\qquad S_2=u_{s, d-\delta+1-t} +\sum_{j\neq t} u_{s,j} u_{s, d-\delta+1-j}.
\]
(If $\delta_*=1$, i.e. $\delta=1$ or $d-\delta=1$, then one of the sums is equal to $u_{s, t}=1$.) 
If $\delta_*\geq 2$, we see that $u\mapsto z_{st}$, $x\mapsto -S_1/2$, $y\mapsto S_2/2$, defines a smooth morphism
\[
V_{s, t}\cap \wti {\mathscr D}_{\rm T}\tol \Spec(\br\O[u, x, y]/(u^2xy-\pi)).
\]
If $\delta_*=1$, we similarly 
obtain a smooth morphism to $\Spec(\br\O[u, x]/(u^2x-\pi))$.
\end{proof}

 \section{Resolution and the linked quadric}\label{s5}
 
 Here, we relate the blow-up $\Mbl(\La)$ of the local model with the linked quadric $\cQ(\La,\La^\vee)$ by introducing a third 
 auxiliary scheme $\calM(\La)$.
  
\subsection{A resolution via additional lines}\label{ss45}

We first define a scheme $\calM^{\rm naive}=\calM^{\rm naive}(\La)$ over $\Spec(\O)$ with $\calG$-action and a $\calG$-equivariant morphism  $r: \calM^{\rm naive}(\La)\tol {\rm M}^{\rm naive}(\La)$. 

We give $\calM^{\rm naive}(\La)$ as a functor on $\O$-algebras as follows. 

\begin{Definition} {\rm The functor $\calM^{\rm naive}(\La)$ associates to an $\O$-algebra $R$ the set of
triples $(W_+, W_-, \calL)$, where
\begin{itemize}
\item[1)]  $\calL$ is a finitely generated projective $R[u]$-module which gives a point of ${\rm M}^{\rm naive}(\La)$, i.e. satisfies conditions (1)-(6) of \S \ref{ss43}, and 

\item[2)] $W_+$, $W_-$, are two finitely generated projective $R[u]$-modules such that
  \[ 
	\xy
	(-21,7)*+{(u-\pi)\mathbb{L}_R \subset   W_+};
	(-4.5,3.5)*+{\rotatebox{-45}{$\,\, \subset\,\,$}};
	(-4.5,10.5)*+{\rotatebox{45}{$\,\, \subset\,\,$}};
	(0,14)*+{\mathbb{L}_R};
	(0,0)*+{\mathcal{L}_R};
	(4,3.5)*+{\rotatebox{45}{$\,\, \subset\,\,$}};
	(4,10.5)*+{\rotatebox{-45}{$\,\, \subset\,\,$}};
	(23.5,7)*+{W_{-} \subset   (u-\pi)^{-1}\mathbb{L}_R,};
	\endxy    
	\]
 with all the quotients   $\bbL_R/W_+$, $\calL/W_+$, $W_{-}/\bbL_R$, $W_{-}/\calL$, finitely generated projective $R$-modules of rank $1$.
 \end{itemize}}
 \end{Definition}
 
 We can see that the forgetful morphism $r: \calM^{\rm naive}(\La)\tol {\rm M}^{\rm naive}(\La)$ given by
 \[
 r(W_+, W_-, \calL)=\calL
 \]
 is representable by a projective morphism. Indeed, $\calM^{\rm naive}(\La)$ is naturally a closed subscheme of $\bbP(\La)\times \bbP(\La^\vee)\times {\rm M}^{\rm naive}(\La)$ and $r$ is given by the projection. Since ${\rm M}^{\rm naive}(\La)$ is projective by Proposition \ref{Mnaive}, $\calM^{\rm naive}(\La)$ is also represented by a projective scheme over $\Spec(\O)$.
 
 This map $r$ is an isomorphism over the open locus ${\rm M}^0(\La)\subset {\rm M}^{\rm naive}(\La)$ where $\calL\neq \bbL$. 
 Indeed, over ${\rm M}^0$ we have $W_+=\calL\cap \bbL$, $W_-=\calL+\bbL$, and the data $(W_+, W_-, \calL)$ are uniquely determined by $\calL$. 
  
 Set $\calM^0(\La)=r^{-1}({\rm M}^0(\La))$ which, by the above, is isomorphic to ${\rm M}^0(\La)$ via $r$.
 For the generic fibers we have 
 \[
 \calM^0(\La)\otimes_\O F={\rm M}^0(\La)\otimes_\O F=\cQ(V).
 \]
 
 \begin{Definition} {\rm
 The scheme $\calM=\calM(\La)$ is the (reduced) Zariski closure of the generic fiber $\cQ(V)=\calM^{0}(\La)\otimes_\O F$ in $\calM^{\rm naive}(\La)$. }
 \end{Definition}
 
 The restriction of the morphism $r$ to $\calM(\La)$ factors through $\Mloc(\La)\subset {\rm M}^{\rm naive}(\La)$ and gives a projective birational $\calG$-equivariant morphism
\[
r: \calM(\La)\tol \Mloc(\La).
\]
We can identify $\calM(\La)$ with both the blow-up $\Mbl(\La)$ of the local model  and the blow-up $\wti \cQ(\La, \La^\vee)$ of the linked quadric $\cQ(\La, \La^\vee)$:
 
 \begin{thm}\label{mainthm}
  There are $\calG$-equivariant isomorphisms $\Mbl(\La)\simeq \calM(\La)\simeq \wti \cQ(\La, \La^\vee)$. 
 \end{thm}
 
 In fact, we obtain a diagram of $\calG$-equivariant birational projective morphisms
 \begin{equation}\label{eqDiagram}
\begin{tikzcd}
&\Mbl(\La) \arrow[dl, "r^{\rm bl}"'] \arrow[r, "\sim"]&\calM(\La) \arrow[dr, "\beta"]\arrow[dll, "r"] &\\
\Mloc(\La) &&&  \cQ(\La, \La^\vee).
\end{tikzcd}
\end{equation}
and the isomorphism $\Mbl(\La)\simeq \calM(\La)$ makes the diagram commute.

 We now give the proof of Theorem \ref{mainthm}.

 \subsection{Comparing resolutions}\label{ss51}
 
We first define a morphism 
 \[
 \rho: \calM(\La)\tol \cQ(\La, \La^\vee)
 \]
 which we then show identifies $\rho$ with the blow-up 
 \[
 \beta: \wti \cQ(\La, \La^\vee)=\wti \cQ_i(\La, \La^\vee)\tol \cQ(\La, \La^\vee).
 \] 
 
  Let us consider an $R$-point of $\calM(\La)$ given by $(W_+, W_-, \calL)$.
 Taking duals gives
 \[ 
	\xy
	(-21,7)*+{(u-\pi)\bbL^\vee_R \subset   W_{-}^\vee};
	(-4.5,3.5)*+{\rotatebox{-45}{$\,\, \subset\,\,$}};
	(-4.5,10.5)*+{\rotatebox{45}{$\,\, \subset\,\,$}};
	(0,14)*+{\bbL^\vee_R};
	(0,0)*+{\calL^\vee};
	(4,3.5)*+{\rotatebox{45}{$\,\, \subset\,\,$}};
	(4,10.5)*+{\rotatebox{-45}{$\,\, \subset\,\,$}};
	(23.5,7)*+{W_{+}^\vee \subset   (u-\pi)^{-1}\bbL^\vee_R.}
	\endxy    
	\]
 We now set
 \[
  (0)\subset \bar W_+\buildrel{1}\over\subset \bbL_R/(u-\pi)\bbL_R=\Lambda_R, 
  \qquad (0)\buildrel{1}\over\subset \bar W_{-}\subset (u-\pi)^{-1}\bbL_R/\bbL_R=\Lambda_R,
  \]
\[
  (0)\subset \bar W_{-}^\vee \buildrel{1}\over\subset \bbL^\vee_R/(u-\pi)\bbL^\vee_R=\Lambda^\vee_R, \qquad (0)\buildrel{1}\over\subset \bar W_{+}^\vee\subset (u-\pi)^{-1}\bbL^\vee_R/\bbL^\vee_R=\Lambda^\vee_R,
  \]
where we denote by bar the   images of the submodules $W_+$, $W_-$, $W_+^\vee$, $W_-^\vee$, in the corresponding quotients.

There is a morphism 
\[
\rho: \calM(\La)\tol \bbP(\La)\times_\O \bbP(\Lambda^\vee)
\]
given by $(W_+, W_-, \calL)\mapsto (\bar W_{-},  \bar W_{+}^\vee)$. Recall that the generic fiber $\calM(\La)\otimes_\O{F}$ is isomorphic 
to the quadric of isotropic lines in the quadratic space $V$. Since $\calM$ is by definition, flat over $\Spec(\O)$, 
and $\cQ(\La, \La^\vee)$ is, also by definition, the flat closure of the same quadric in $\cQ(\La)\times_\O \cQ(\La^\vee)\subset \bbP(\La)\times_\O \bbP(\La^\vee)$, the morphism $\rho$ factors through $\cQ(\La, \La^\vee)$ to give
\[
\rho: \calM(\La)\tol \cQ(\La, \La^\vee).
\]
We will use $\rho$ to identify $\calM(\La)$ with the blow-up $\wti \cQ(\La,\La^\vee)$ of $\cQ(\La, \La^\vee)$.

\subsection{Proof of the main comparison}\label{ss52}

We continue with the proof of Theorem \ref{mainthm}. For simplicity, we will omit $\La$ from some of the 
notation and write $\calM$, $\Mloc$, etc.
We have already given morphisms $r: \calM\to \Mloc$ and $\rho:\calM\to \cQ(\La,\La^\vee)$.
We would like to show that these induce identifications with the blow-up $r^{\rm bl}: \Mbl\to \Mloc$.
Using the universal property of the blow-up, we see that it is enough to
prove this statement after base changing to $\O_{F'}$, where $F'/F$ is an unramified extension, or even to $\br \O$. This allows us to assume that we are in one of the $4$ cases listed in \S \ref{ss42a}. To ease up on the notation
we will omit this base change in what follows.

Recall that $\calM$ is a closed subscheme of
$\bbP(\Lambda)\times\bbP(\Lambda^\vee)\times \Mloc$. This can also be seen as follows: Start with
\[
(x_1;\ldots ; x_d)\times (y_1;\ldots ; y_d) \in \bbP(\Lambda)\times \bbP(\Lambda^\vee).
\]
Take $W=W_+$ by giving $\bar W_+\subset \Lambda_R$ to be the perpendicular of the line $(y_1;\ldots ; y_d)\in \bbP(\Lambda^\vee)$
under the perfect pairing $\Lambda_R\times \Lambda^\vee_R\to R$. Similarly, we take $W'=W_{-}^\vee$ by giving $\bar W'\subset \Lambda^\vee_R$ to be the perpendicular of the line $(x_1;\ldots ; x_d)\in \bbP(\Lambda)$
under the perfect pairing $\Lambda_R\times \Lambda^\vee_R\to R$. 
Set $e_i^\vee$ for the dual basis of $e_i\in \La$ so that
\[
\lan e_i, e^\vee_j\ran=\delta_{ij}.
\]
Then we have
\[
\bar W_+^\vee=(\sum_i y_i e_i^\vee)\subset \Lambda^\vee_R,\qquad
\bar W_{-}=(\sum_i x_i e_i)\subset \Lambda_R.
\]
The pair $(W_+, W_-)$ is part of a triple $(W_+, W_-, \calL)$ that corresponds to a point of $\calM^{\rm naive}$, when there is $\calL\in {\rm M}^{\rm naive}(R)$, such that
\begin{itemize}
\item[1)] the image of $\calL$ in $(u-\pi)^{-1}\bbL_R/\bbL_R=\Lambda_R$, i.e. the image of $\Phi$, is contained
in $\bar W_{-}=(\sum_i x_i e_i)$, and
\item[2)] the image of $\calL^\vee$ in $(u-\pi)^{-1}\bbL^\vee_R/\bbL^\vee_R=\Lambda^\vee_R$
is contained
in $\bar W^\vee_{+}=(\sum_i y_i e_i^\vee)$. 
\end{itemize}

Let us now consider the inverse image $\wti U:=r^{-1}(U)$ under
\[
r: \calM\tol \Mloc.
\] 
Recall $U=U^{\rm naive}\cap \Mloc$.
Over the affine chart $U^{\rm naive}$, $\Phi$ is given by the matrix $X$.
Hence,
 condition (1) translates to
\begin{equation}\label{mat.rel1}
X=(x_1,\ldots, x_d)^t\cdot (\lambda_1,\ldots , \lambda_d)=\left( \lambda_1\underline x^t |\cdots |\lambda_d\underline x^t\right)
= \begin{pmatrix} 
         \lambda_{1}x_1  & \dots & \lambda_{d} x_1\\
         \vdots & \ddots & \vdots   \\
           \lambda_{1} x_{d} & \dots & \lambda_{d} x_{d}   \\
    \end{pmatrix} 
\end{equation}
for some $(\lambda_1,\ldots , \lambda_d)\in R^d$.
Similarly,   condition (2) translates to
\begin{equation}\label{mat.rel1b}
Y=(y_1,\ldots, y_d)^t\cdot (\mu_1,\ldots , \mu_d)=\left( \mu_1\underline y^t |\cdots |\mu_d\underline y^t\right)
= \begin{pmatrix} 
         \mu_{1}y_1  & \dots & \mu_{d} y_1\\
         \vdots & \ddots & \vdots   \\
           \mu_{1} y_{d} & \dots & \mu_{d} y_{d}   \\
    \end{pmatrix}
\end{equation}
for some $(\mu_1,\ldots , \mu_d)\in R^d$. 

To continue with our calculation, it is convenient to recall the decompositions $\La=M\oplus N$, $\La^\vee=M\oplus \pi^{-1}N$ as in \S \ref{ss42a} and the bases of $M$, $N$ 
listed there. Recall that we assume that we are in one of the $4$ cases of \S \ref{ss42a}. We define $\Delta$, resp. $\Delta^c$, to be the set of $1\leq i\leq d$ with $e_i\in N$, resp. $e_i\in M$, where $\{e_i\}$ is the basis listed there. Then $\{1,\ldots , d\}=\Delta^c\sqcup \Delta$ and   $\#\Delta=\delta$, $\#\Delta^c=d-\delta$.  For $\und x=(x_i)_{1\leq i\leq d}\in \bbA(\La)$ we write
\[
\und x=\und x_{1}+\und x_{2}, 
\]
  where $ \und x_1\in \bbA(M)$, $ \und x_2\in \bbA(N)$. Let $\und w=(w_j)_{j\in \Delta^c}$ be a point of 
$\bbA(M)$, resp. $\und z=(z_i)_{i\in \Delta}$ a point of $\bbA(N)$. As in the proof of Theorem \ref{thmQ}, we set 
\[
Q_1(\und w) =\frac{\lan \und w, \und w\ran}{2},\qquad
Q_2(\und z)=\frac{\lan \und z, \und z\ran}{2\pi}.
\]

Let $\wti U_{s,t}$, where $x_s=1$ and $y_t=1$, be the affine patches that cover $\wti U$. Using the equation $Y+X^t=0$ we obtain the following relations:
\begin{equation}\label{lam.eq1}
  \lambda_t = - \mu_{s}, \quad  \lambda_i = \lambda_t y_i, \quad \text{and} \quad \mu_j = -\lambda_t x_j \quad \text{for} \quad 1 \leq i,j \leq d.
  \end{equation}

We can now determine $\wti U_{s, t}$.

 \begin{prop}\label{affineChart}
  We have
\begin{equation}\label{eqUstA}
\wti U_{s, t} \simeq  \Spec\left(\frac{\O[\lambda_t, (x_i)_{1 \leq i \leq d},(y_j)_{1 \leq j \leq d}]}{( \lambda^2_tQ_2(\und x_2) Q_1( \und y_1 )+ \pi)+\mathfrak{K}_{s, t}} \right)
\end{equation}
where  
\[
\mathfrak{K}_{s, t} =  \Bigg( x_s-1, \    y_t-1, \
(x_{i}+\lambda_t {Q_2(\und x_2)}y_{d+1-i} )_{i \in \Delta^c},  \  (y_j -  \lambda_t {Q_1( \und y_1 )}x_{d+1-j} )_{j \in \Delta}  \Bigg).
\]
\end{prop}
 
 Before we give the proof, we note that this implies that the charts $\wti U_{s, t}$ with $s\in\Delta$, $t\in \Delta^c$,
 cover $\wti U$. Indeed, if $x_i=0$ on $\wti U_{s, t}$, for all $i\in \Delta$, then also $Q_2(\und x_2)=0$. Hence, since
 $x_{i}+\lambda_t {Q_2(\und x_2)}y_{d+1-i}\in \mathfrak{K}_{s, t}$, we obtain $x_i\in \mathfrak{K}_{s, t}$ for all $i\in \Delta^c$ also, which is a contradiction. 
 Hence, we also have $x_i\neq 0$ for some $i\in \Delta$.
 A similar argument gives $y_j\neq 0$ for some $j\in \Delta^c$.
 If $s\in \Delta$ and $t\in \Delta^c$, then by eliminating $x_i$, for $i\in \Delta^c$, and $y_j$, for $j\in \Delta$, using the relations given by the generators of $\mathfrak{K}_{s, t}$, we obtain
 \begin{equation}\label{eqUst}
\wti U_{s, t} \simeq \Spec \left(\frac{\O[\lambda_t, (x_i)_{i\in \Delta},(y_j)_{j\in \Delta^c}]}{ (\lambda^2_tQ_2(\und x_2) Q_1( \und y_1 )+ \pi, x_s-1, y_t-1)} \right).
\end{equation}

\begin{proof}
 We will assume $d$ and $\delta$ have the same parity and leave the very similar (but notationally more involved)
 case of different parity to the reader. Set $l=(d-\delta)/2=n-r$. Recall the decomposition 
 \[
 X=(x_1,\ldots, x_d)^t\cdot (\lambda_1,\ldots , \lambda_d)= \begin{pmatrix} 
         \lambda_{1}x_1  & \dots & \lambda_{d} x_1\\
         \vdots & \ddots & \vdots   \\
           \lambda_{1} x_{d} & \dots & \lambda_{d} x_{d}   \\
    \end{pmatrix} 
=\left[\ 
\begin{matrix}[c|c|c]
D_1 & C_1 & D_2 \\ \hline
B_1 & A & B_2 \\ \hline
D_3 & C_2 & D_4
\end{matrix}\ \right].
 \]
 The block decomposition corresponds to separating a vector $\und x$ into $3$ parts
 \[
 \und x=\und x_1\oplus \und x_2=\und x_{1-}\oplus \und x_2\oplus \und x_{1+}\ ,
 \]
 in spaces of dimension $l$, $\delta$, $l$, respectively,
 with
 \[
\und  x_{1-}=(x_1,\ldots, x_l),\quad \und x_2=(x_{l+1},\ldots , x_{l+\delta}),\quad \und x_{1+}=(x_{l+\delta+1},\ldots , x_d).
 \]
 Denote by $\und x_2^*$, resp. $\und x_{1\pm}^*$,  the result of reversing the order of the coordinates in the vector $\und x_2$, resp. $\und x_{1\pm}$.  We have
 \[
 A=\und x_2^t\cdot \und\lambda_2=\begin{pmatrix} 
         \lambda_{l+1}x_{l+1}  & \dots & \lambda_{l+\delta} x_{l+1}\\
         \vdots & \ddots & \vdots   \\
           \lambda_{l+1} x_{l+\delta} & \dots & \lambda_{l+\delta} x_{l+\delta}
           \end{pmatrix} ,
 \]
 \[
 B_1 =  \und x_2^t\cdot \und \lambda_{1-}=\begin{pmatrix} 
         \lambda_{1}x_{l+1}  & \dots & \lambda_{l}x_{l+1} \\
         \vdots & \ddots & \vdots   \\
          \ \lambda_{1} x_{l+\delta}  & \dots & \lambda_{l} x_{l+\delta} \\
         
    \end{pmatrix} , \ \quad 
    B_2= \und x_2^t\cdot \und\lambda_{1+}=\begin{pmatrix} 
         \lambda_{l+\delta+1}x_{l+1}  & \dots & \lambda_{d}x_{l+1} \\
         \vdots & \ddots & \vdots   \\
          \ \lambda_{l+\delta+1} x_{l+\delta}  & \dots & \lambda_{d} x_{l+\delta}, \\
    \end{pmatrix}.
\]
Similarly, $D_1=\und x_{1-}^t\cdot\und\lambda_{1-}$, etc.
We have
\begin{eqnarray}\label{eqbbj}
B_2 J_{l} B_1^{t}J_{\delta}&=& \und x_2^t\cdot \und\lambda_{1+}\cdot J_l \cdot (\und x_2^t\cdot \und \lambda_{1-})^t\cdot J_\delta \ \ \\
 &=&
\und x_2^t\cdot (\und\lambda_{1+}\cdot J_l \cdot\lambda_{1-}^t)\cdot \und x_2\cdot J_\delta \nonumber \\
&=&\und x_2^t \cdot \und x_2^* \cdot Q_1(\und \lambda_1),\nonumber
\end{eqnarray}
since $Q_1(\und \lambda_1)=\und\lambda_{1+}\cdot J_l \cdot\lambda_{1-}^t$ and $\und x_2\cdot J_\delta=\und x_2^*$.
The relation ${\rm Tr}(B_2 J_{l} B_1^{t}J_{\delta}) + 2 \pi=0$, that holds over $U$ by Theorem \ref{Za1}, translates to
\begin{equation}\label{q1q2}
Q_2(\und x_2)\cdot  {Q_1(\und \lambda_1)}+\pi=0.
\end{equation}
 Using the relations (\ref{lam.eq1}) we obtain:
\begin{equation}\label{equation1}
    \lambda^2_t Q_1(\und y_1) Q_2(\und x_2) +\pi=0.
\end{equation}
Notice that this last equation implies that $\lambda_t$ is not a zero divisor in the coordinate ring of the $\O$-flat scheme
$\wti U_{s, t}$. Since, by (\ref{lam.eq1}), $\lambda_i=\lambda_t y_i$, the relation $A= B_2 J_{n-r} B_1^{t}J_{\delta}$,
and (\ref{eqbbj}) gives
\[
\lambda_t\cdot \und x_2^t\cdot \und y_2=\und x_2^t \cdot \und x_2^* \cdot \lambda_t^2 Q_1(\und y_1).
\]
 Since $\lambda_t$ is not a zero divisor, we obtain
 \begin{equation*}
 \und x_2^t\cdot \und y_2= \und x_2^t \cdot \und x_2^* \cdot \lambda_t  Q_1(\und y_1)
 \end{equation*}
 or, after taking transpose,
 \begin{equation}\label{equation2}
  (\und y_2-    \lambda_t  Q_1(\und y_1)\cdot \und x_2^*)^t\cdot \und x_2=0.
 \end{equation}
 By a similar calculation as in the proof of (\ref{eqbbj}), we obtain 
\[
\frac{1}{2} J_{l} B_2^{t}J_{\delta} B_1 =(\und \lambda^*_{1+})^t\cdot \und \lambda_{1-}\cdot  {Q_2(\und x_2)}.
\quash{ \begin{pmatrix} 
      \lambda_{1}   \lambda_{d} & \cdots & \lambda_{l} \lambda_{d}    \\
         \vdots & \ddots & \vdots   \\
          \lambda_{1}  \lambda_{l+\delta+1}  & \cdots & \lambda_{l} \lambda_{l+\delta+1}    \\
\end{pmatrix}\cdot {Q_2(\und x_2)}.}
\]
Hence, the relation $D_1 + \frac{1}{2} J_{l} B_2^{t}J_{\delta} B_1=0$ which holds over $U\subset \Mloc$ (see the proof of Theorem \ref{Za1}) amounts to
\[
 \und x_{1-}^t\cdot \und \lambda_{1-}=\lambda_t\cdot \und x_{1-}^t\cdot \und y_{1-}=-\lambda^2_t\cdot (\und y^*_{1+})^t\cdot \und y_{1-}\cdot  {Q_2(\und x_2)},
\]
and over the $\O$-flat $\wti U_{s, t}$ to
\begin{equation}\label{equation3}
 \und x_{1-}^t\cdot \und y_{1-}=-(\und y^*_{1+})^t\cdot \und y_{1-}\cdot \lambda_t {Q_2(\und x_2)}.
\end{equation}

Equivalently, this is
\begin{equation}\label{equation3*}
(\und x_{1-} + \lambda_t\cdot  {Q_2(\und x_2)}\cdot \und y^*_{1+})^t\cdot \und y_{1-}=0.
\end{equation}
Similarly, from $ C_1 = -\frac{1}{2} J_{l} B_2^{t}J_{\delta} A$, $D_2 = -\frac{1}{2} J_{l} B_2^{t}J_{\delta} B_2$, we obtain
\begin{equation}\label{equation4*}
(\und x_{1-} + \lambda_t\cdot  {Q_2(\und x_2)}\cdot \und y^*_{1+})^t\cdot \und y_{2}=0,
\end{equation}
\begin{equation}\label{equation5*}
(\und x_{1-} + \lambda_t\cdot  {Q_2(\und x_2)}\cdot \und y^*_{1+})^t\cdot \und y_{1+}=0.
\end{equation}
Since $y_t=1$, these, all together, amount to
\begin{equation}\label{minus}
\und x_{1-} + \lambda_t\cdot  {Q_2(\und x_2)}\cdot \und y^*_{1+}=0.
\end{equation}
We now examine the relations $C_2 = -\frac{1}{2} J_{l} B_1^{t}J_{\delta} A $,
$ D_3 = -\frac{1}{2} J_{l} B_1^{t}J_{\delta} B_1$, $D_4 = -\frac{1}{2} J_{l} B_1^{t}J_{\delta} B_2 $.  In a similar fashion, 
we find that these, all together,  amount to
\begin{equation}\label{plus}
\und x_{1+} + \lambda_t\cdot  {Q_2(\und x_2)}\cdot \und y^*_{1-}=0.    
\end{equation}
Combining the relations (\ref{minus}) and (\ref{plus}) we obtain the desired equations
\[
x_{i}+\lambda_t  {Q_2(\und x_2)} y_{d+1-i}=0, \ \forall i\in \Delta^c.
\]
Finally, (\ref{equation2}) amounts to
\[
x_i(y_j -  \lambda_t  {Q_1( \und y_1 )} x_{d+1-j})=0, \ \forall i, j\in \Delta.
\]
By (\ref{equation1}), $Q_2(\und x_2)$ is not a zero divisor in the coordinate ring of
the $\O$-flat scheme
$\wti U_{s, t}$. Since $Q_2(\und x_2)$ belongs to the ideal $(x_i)_{i\in \Delta}$, we see that over 
$\wti U_{s, t}$, we also have
\[
y_j -  \lambda_t  {Q_1( \und y_1 )} x_{d+1-j}=0, \ \forall j\in \Delta.
\]
Assembling the above, we see that all the generators of the ideal $\mathfrak K_{s, t}$ vanish on 
$\wti U_{s, t}$. Therefore, $\wti U_{s,t}$ is a closed subscheme of the spectrum of the ring that appears 
on the right hand side of (\ref{eqUstA}). We can see by eliminating $x_i$, for $i\in \Delta^c$, and $y_j$, for $j\in \Delta$, that this spectrum is a hypersurface in affine space of relative dimension $d-1$ and is then integral of dimension $d-1$. Since also $\dim(\wti U_{s, t})=d-1$, it now follows, by comparing dimensions, that the equaility
(\ref{eqUstA}) holds, and so $\wti U_{s, t}$ is as in the statement.
\end{proof}
\smallskip
 
We now continue on to show that there is a $\calG$-equivariant isomorphism 
$\calM\simeq \wti \cQ(\La, \La^\vee)$. First, we show that $ \rho: \calM\to \cQ(\La, \La^\vee)$ factors through the blow-up 
\begin{center}
\begin{tikzcd}
\calM \arrow[r, rightarrow] \arrow[dr, "\rho"']
& \wti \cQ(\La, \La^\vee) \arrow[d, "\beta"]\\
& \cQ(\La, \La^\vee).
\end{tikzcd}
\end{center}
It is enough, by the universal property of the blow-up, to show that the pull-back $\rho^*(Z_1)$ of $Z_1\subset \cQ(\La,\La^\vee)$ 
to $\calM$ is a Cartier divisor. Recall that $Z_1$ is locally defined by the ideal $(u,Q_2(\und x_2))$, where $u$ is the ``linking multiplier",
defined up to unit by $i_R(L)=u L'$, where $L$ and $L'$ are the universal isotropic lines over $\cQ(\La, \La^\vee)$. We will show that
over each affine chart $\wti U_{s, t}\subset \calM$ the linking multiplier is, in fact, a multiple of $Q_2(\und x_2)$. Hence, the pull-back of the ideal
$(u,Q_2(\und x_2))$ is principal, as desired.  

In Proposition \ref{affineChart} above, we are using the dual basis $e_i^\vee$ of $\Lambda^\vee$, which by definition, 
is given by $\lan e_i, e^\vee_{j}\ran=\delta_{ij}$. We have
\[
e_i^\vee = e_{d+1-i} \quad \text{for} \quad i \in S(M), \quad \text{and} \quad  e_j^\vee=\pi^{-1}e_{d+1-j} \quad \text{for} \quad  i\in S(N). 
\]
 With this basis, the ``linking" 
\[
i_R: \La_R\to \La^\vee_R
\]
is given by the symmetric matrix $S=(\lan e_i, e_j\ran)_{i, j}$. 
For example, in the case $d$ and $\delta$ have the same parity, we have
\[
i_R(\und x_{1-}, \und x_2, \und x_{1+})=(\und x_{1-}, \pi\und x_2, \und x_{1+}).
\]
 Hence, in view of the elements generating $\mathfrak K_{s, t}$, we see that Proposition \ref{affineChart} implies that over $\wti U_{s, t}\subset \calM$ we have
 \[
 i_R(\und x)=u\cdot \und y, \quad j_R(\pi \und y)=v\cdot \und x,
 \]
with 
\begin{equation}\label{l.con}
u = -{Q_2(\und x_2) }\lambda_t ,   \qquad v= {Q_1( \und y_1 )} \lambda_t  .    
\end{equation}
This establishes that the pullback of the ideal $(u, Q_2(\und x_2))$ is principal over $\wti U_{s, t}$ and so, locally principal over $\wti U$. Therefore, the restriction $\rho_{|\wti U}: \wti  U\to \cQ(\La, \La^\vee)$ factors through the blow-up $\wti \cQ(\La,\La^\vee)\to \cQ(\La,\La^\vee)$. The result for $\rho$ easily follows since $\rho$ is $\calG$-equivariant. Indeed, the $\calG$-translates of $U$ cover $\Mloc$. Hence, the $\calG$-translates of the open $\wti U=r^{-1}(U)\subset \calM$ cover $\calM$. This, combined with the above, implies that $\rho$ factors through $\wti\rho: \calM\to \wti \cQ(\La, \La^\vee)$. 

It remains to show that $\wti\rho: \calM\to \wti \cQ(\La, \La^\vee)$ is an isomorphism. This is easily obtained by using the description of the affine charts given in Proposition \ref{affineChart}, and the discussion above
together with $\calG$-equivariance: The map $\wti\rho$ is birational.
From Proposition \ref{affineChart} and the usual $\calG$-equivariance argument, $\calM$ is a regular scheme and is projective and flat over $\Spec(\O)$ of relative dimension $d-2$. The same is true for $\wti \cQ(\La, \La^\vee)$  by Proposition \ref{prop331}. In fact, by comparing the explicit description of the affine charts $\wti U_{s, t}\subset \calM$ given by Proposition \ref{affineChart}, with that of the affine charts for $\wti \cQ(\La, \La^\vee)$ given in the proof of Proposition \ref{prop331}, we can easily see, using (\ref{l.con}) above,  that $\wti\rho$ gives a bijection on $\bar k$-points.  Hence, the morphism $\wti\rho$ is birational quasi-finite and then, by using Zariski's main theorem, an isomorphism. This concludes the proof of the existence of a $\calG$-isomorphism $\calM\simeq \wti\cQ(\La,\La^\vee)$.

It remains to give a $\calG$-equivariant isomorphism $\calM\xrightarrow{\sim} \Mbl$. Again, we first show that
$r: \calM\to \Mloc$ factors through the blow-up $r^{\rm bl}: \Mbl\to \Mloc$. By using $\calG$-equivariance we see it is enough to show that the pull-back of the ideal generated by the entries of the matrix $X$ becomes principal over $r^{-1}(U)=\wti U\subset \calM$. This follows immediately from the equations in the proof of Proposition \ref{affineChart}. The proof of the fact that the map $\calM\to \Mbl$ is an isomorphism  follows
the same line of argument as for the map $\wti\rho$ above. In fact, after unravelling the various identifications of coordinate systems, we can see that the description of the affine chart $\wti U_{s, t}$ in (\ref{eqUst}) matches the description of the corresponding affine chart $V_{s, t}\cap U^{\rm bl}$ of the blow up $U^{\rm bl}$ given by (\ref{quadBU}). Hence, $\calM\to \Mbl$ restricts to an isomorphism $\wti U_{s, t}\xrightarrow{\sim}  V_{s, t}\cap U^{\rm bl}$. This concludes the proof of Theorem \ref{mainthm}. \hfill $\square$

 \subsection{Additional properties}\label{ss53}
 
 We now give some further properties:
  \quash{\begin{equation}\label{eqDiagram2}
\begin{tikzcd}
&\calM\arrow[dl, "r"']\arrow[r, "\simeq"]\arrow[dr, "\rho"]  &\wti \cQ(\La, \La^\vee)\arrow[d, "\beta"] &\\
\Mloc &&  \cQ(\La, \La^\vee). &
\end{tikzcd}
\end{equation}
}
 
 \begin{prop}\label{propAdditional}
 a)  The exceptional locus of $r: \calM(\La)=\Mbl(\La)\longrightarrow \Mloc(\La)$ is 
  \[
 r^{-1}(*)= \bbP(\Lambda^\vee/\La)\times_k \bbP(\La/\pi\La^\vee)\simeq \bbP^{\delta-1}\times_k \bbP^{d-\delta-1}.
 \]
 
 b)  The morphism $\rho: \calM(\La)\to \cQ(\La, \La^\vee)$  is an isomorphism over the complement of the intersection $Z_0\cap Z_1\cap Z_2=\cQ(\Lambda^\vee/\La)\times_k \cQ(\La/\pi\La^\vee)\subset \cQ(\La, \La^\vee)$. Over this intersection $\rho$ is a $\bbP^1$-bundle. 
 \end{prop}
 \begin{proof}
  Let us describe the inverse image $r^{-1}(*)\subset \calM(\La)$:
 We easily see that $r^{-1}(*)$ is a closed subscheme of 
 \[
 \{(x_1;\ldots ; x_d), (y_1;\ldots ; y_d)\}= \bbP(\La)_k\times_k \bbP(\La^\vee)_k.  
 \]
 Over the intersection $
 \wti U_{s, t}\cap r^{-1}(*)
 $
 we have $x_s=1$ and $X=\und x^t\cdot \und \lambda=0$. This gives $\und\lambda=0$ and in particular
 $\lambda_t=0$. Using the equations for $\wti U_{s, t}$ given by Proposition \ref{affineChart}, we see that 
 $
 \wti U_{s, t}\cap r^{-1}(*)
 $
 is defined by $x_i=0$ for all $i\in \Delta^c$, $x_s=1$, and $y_j=0$, for all $j\in \Delta$, $y_t=1$. Therefore, the inverse
 image $r^{-1}(*)$ is
 \[
 \bbP(N)_k\times_k \bbP(M)_k=\bbP(\Lambda^\vee/\La)\times \bbP(\La/\pi\La^\vee)\simeq \bbP^{\delta-1}\times_k \bbP^{d-\delta-1}.
 \]
 This proves (a). (Alternatively, we could have used the description of $r$ as a blow-up from Theorem \ref{mainthm} and its proof.) 
 Part (b) follows from the description of the blow-up $\wti \cQ(\La, \La^\vee)\to \cQ(\La, \La^\vee)$ in Proposition \ref{prop331} and its proof,
 and Theorem \ref{mainthm} which identifies $\rho$ with this blow-up.
 \end{proof}

 \begin{Remark}
 {\rm a) We can now explain the birational map 
 \[
  \cQ(\La, \La^\vee) \dashrightarrow \Mloc(\La)
  \]
  as follows:  We first perform the blow-up 
 $\wti \cQ(\La, \La^\vee)$ of $\cQ(\La, \La^\vee)$. Then to obtain $\Mloc(\La)$, 
 we contract a subscheme $\wti Z_0$, isomorphic to $\bbP^{\delta-1}\times \bbP^{d-\delta-1}$ (which is in fact an irreducible
 component of the special fiber of the blow-up) to a point. This subscheme is the strict transform of the component 
 $Z_0\simeq \bbP^{\delta-1}\times \bbP^{d-\delta-1}$ of the special fiber of $\cQ(\La,\La^\vee)$. \quash{Note that the blow-up $\beta: \wti \cQ(\La, \La^\vee)\to \cQ(\La, \La^\vee)$  is an isomorphism over the complement of the intersection $Z_0\cap Z_1\cap Z_2=\cQ(\Lambda^\vee/\La)\times_k \cQ(\La/\pi\La^\vee)\subset \cQ(\La, \La^\vee)$. Over this intersection $\beta$ is a $\bbP^1$-bundle. }
 
 b) The  cases $\delta=0$ and $\delta=1$ are different. 
 
 i) When $\delta=0$, $\La=\La^\vee$ and $\Mloc(\La)$ is the smooth quadric $\cQ(\La)$. Then 
 \[
 \Mloc=\wti \cQ(\La, \La^\vee)=\cQ(\La,\La^\vee).
 \]
 
 ii) When $\delta=1$, we also have $\Mloc(\La)\simeq \cQ(\La)$ (see \cite[Prop. 12.7]{HPR}). This case was also discussed in detail, but from slightly different perspectives, in \cite[12.7.2]{HPR} and in \cite[Prop. 2.16]{MP}. Now, $\Mloc(\La)$ is not smooth over $\Spec(\O)$ but only regular;
 the special fiber has an isolated singular point. In this case, $\cQ(\La, \La^\vee)$ (which is denoted by $P(\La)^{\rm fl}$ in \cite[12.7.2]{HPR}) is a blow-up of $\cQ(\La)$ at this singular point. We have $\wti \cQ(\La, \La^\vee)=\cQ(\La, \La^\vee)$,  $\rho: \calM(\La)\to \cQ(\La, \La^\vee)$
 is an isomorphism, and
 \[
 r: \calM(\La)=\Mbl(\La)=\wti \cQ(\La, \La^\vee)=\cQ(\La, \La^\vee)\xrightarrow{ \ } \Mloc(\La)=\cQ(\La)
 \]
 can be identified with the blow-up of $\Mloc(\La)=\cQ(\La)$ at its singular point, discussed above. In particular, $\calM(\La)=\Mbl(\La)$ is isomorphic to
 $P(\La)^{\rm fl}$ of \emph{loc. cit.}
 
 }
 
 \end{Remark}

\section{Application to Shimura varieties}\label{sShimura}

\subsection{Spin and orthogonal Shimura data}\label{ss61}

We now discuss some Shimura varieties to which we can apply these results.  We start with an odd prime $p$ and an orthogonal quadratic space 
$V$ over $\Q$ of dimension $d\geq 5$ and signature $(d-2, 2)$. 

As in the local set-up of \S \ref{ss23}, the Clifford algebra $C(V)$ is endowed with a $\Z/2\Z$-grading 
$
C(V)=C^+(V)\oplus C^-(V)
$
and a canonical involution $c\mapsto c^*$.
The group of spinor similitudes $G=\GSpin(V)$  is the reductive group over $\Q$ defined by
\[
G(R) = \{  g\in C^+(V_R)^\times \ |\ g V_R g^{-1} = V_R,\, g^* g \in R^\times \}
\]
for any $\Q$-algebra $R$.  The spinor similitude $\eta: G \to \Gm$ is  defined by 
$\eta(g) =g^* g$, and there is a representation $\alpha: G \to \SO(V)$ defined by  $g\cdot v = gvg^{-1}$.

Consider the hermitian symmetric domain 
\[
 X= \{ z\in V_\C : \lan z,z\ran=0,\,  \lan z,\bar{z}\ran<0 \} /\C^\times
\]
 of dimension $d-2$. The group $G(\R)$ acts on $X$ through $G \to \SO(V)$,
and the action of any $g\in G(\R)$ with $\eta_G(g) < 0$ interchanges the two connected components of $X$.

Now write   $z\in X$  as  $z=u+iv$ with $u,v\in V_\R$. Then, the subspace $\mathrm{Span}_\R\{u,v\}$
is a negative definite plane in $V_\R$, oriented by the ordered orthogonal basis $u,v$.  There are natural $\R$-algebra maps 
\[
\C \xrightarrow{\sim}  C^+(\mathrm{Span}_\R\{u,v\}) \to C^+(V_\R).
\]
The first is  determined by 
\[
i\mapsto  \frac{uv} { \sqrt{ Q(u) Q(v)} },
\]
and the second is induced by  the inclusion $\mathrm{Span}_\R\{u,v\} \subset V_\R$.  The above composition restricts to an injection
$h_z: \C^\times \to G(\R)$, which arises from a morphism $h_z : \mathbb{S} \to G_\R$ of real algebraic groups.
Here  $\mathbb{S} = \mathrm{Res}_{\C/\R} \Gm$ is  Deligne's torus.  
The construction $z\mapsto h_z$  realizes 
$
X \subset \Hom(\mathbb{S} , G_\R)
$
as a $G(\R)$-conjugacy class. The pair $(G, X)$ is a Shimura datum of Hodge type.

Using the conventions of \cite{DeligneCorvallis}, the Hodge structure on $V$ determined by ${h}_z$ is
\begin{equation}\label{gspin hodge}
V_\C^{(1,-1)} = \C z ,\quad
V_\C^{(0,0)} = ( \C z + \C \bar{z} )^\perp,\quad
V_\C^{(-1,1)} = \C \bar{z}.
\end{equation}

This implies that the Shimura cocharacter $\mu_z: \Gm_\C\to G_\C$ obtained from $\{h_z\}$
is conjugate to $\ti \mu:\Gm_\C\to G_\C$ given by
\[
\ti \mu(t)=t^{-1}f_1f_{d}+f_{d}f_1 
\]
where $f_1$, $f_d$ are part of a basis $(f_i)_i$ of $V_\C$ such that $\lan f_1, f_d \ran=1$, $\lan f_1, f_1\ran=\lan f_d, f_d\ran=0$,
and the arithmetic on the right hand side takes place in the Clifford algebra $C(V_\C)$. This agrees with the cocharacter considered in \S \ref{ssSpinLM} above.
 As in \S \ref{ssSpinLM}, we see that $\alpha\cdot \wti \mu: \Gm_\C\to \SO(V)_\C$ is given by $\mu(t)={\rm diag}(t^{-1}, 1, \ldots , 1, t)$.

Observe that the action of $G(\R)$ on $X$ factors through $\SO(V)(\R)$ via $\alpha$, and we also obtain a Shimura datum $(\SO(V), X)$
which is now of abelian type. The corresponding Shimura cocharacter is conjugate to $\mu(t)={\rm diag}(t^{-1}, 1, \ldots , 1, t)$.

\subsection{Spinor integral models}\label{ss62}

We continue with the notations and assumptions of the previous paragraph. In particular, we take $G=\GSpin(V)$ and $X$
the $G(\R)$-conjugacy class of $\{h_z\}:\mathbb{S}\to G_\R$ above that define the spin similitude Shimura datum $(G, X)$.

In addition, we choose a vertex lattice $\La\subset V\otimes_{\Q}\Q_p$ with $\pi\La^\vee\subset \La\subset \La^\vee$
and $\delta={\rm length}_{\Z_p}(\La^\vee/\La)$, $\delta_*={\rm min}(\delta, d-\delta)$, and assume $\delta_*\geq 1$. This defines the parahoric subgroup
\[
K_p=\{g\in \GSpin(V\otimes_{\Q}\Q_p)\ |\ g\La g^{-1}=\La, \ \eta(g)\in \Z^\times_p\}
\]
which we fix below. 
 Choose also a sufficiently small compact open subgroup $K^p$ of the prime-to-$p$ finite adelic points $G({\mathbb A}_{f}^p)$ of $G$
 and set $K=K^pK_p$. 
 The Shimura variety  ${\rm Sh}_{K}(G, X)$ with complex points
 \[
 {\rm Sh}_{K}(G, X)(\C)=G(\Q)\backslash X\times G({\mathbb A}_{f})/K
 \]
 is of Hodge type and has a canonical model over the reflex field $\Q$.

 \begin{thm}\label{LM}
 For every $K^p$ as above, there is a
 scheme $\mathscr S^{\rm reg}_K(G, X)$, flat over $\Spec(\Z_{p})$, 
 with
 \[
 \mathscr S^{\rm reg}_K(G, X)\otimes_{\Z_p}\Q_p={\rm Sh}_K(G, X)\otimes_{\Q}\Q_p,
 \]
 and which supports a ``local model diagram"
  \begin{equation}\label{LMdiagram}
\begin{tikzcd}
&\wti{\mathscr{S}}^{\rm reg}_K(G, X)\arrow[dl, "\pi^{\rm reg}_K"']\arrow[dr, "q^{\rm reg}_K"]  & \\
\mathscr S^{\rm reg}_K(G, X)  &&  \Mbl(\La)
\end{tikzcd}
\end{equation}
such that:
\begin{itemize}
\item[a)] $\pi^{\rm reg}_K$ is a $\calG$-torsor for the parahoric group scheme $\calG$ that corresponds to $K_p$,

\item[b)] $q^{\rm reg}_K$ is smooth and $\calG$-equivariant.

\item[c)] $\mathscr S^{\rm reg}_K(G, X)$ is regular and has special fiber which is a divisor with
normal crossings. The multiplicity of each irreducible component of the special fiber is either one  or two and the
components
of multiplicity two are each isomorphic to $\bbP^{\delta-1}\times \bbP^{d-\delta-1}$ over $\bar{\mathbb F}_p$.
In fact, $\mathscr S^{\rm reg}_K(G, X)$ can be covered, in the \'etale topology, by schemes which are smooth over $
\Spec(\Z_p[u, x, y]/(u^2xy-p))
$
when $\delta_*\geq 2$, or over
$
\Spec(\Z_p[u, x]/(u^2x-p))
$
when $\delta_*= 1$. 

\end{itemize}

In addition, we have: 

\begin{itemize}
 \item[1)] The schemes  $\{\mathscr S^{\rm reg}_{K}(G, X)\}_{K^p}$, for variable $K^p$,  support correspondences that 
extend the standard prime-to-$p$ Hecke correspondences on $\{{\rm Sh}_{K}(G, X)\}_{K^p}$. These correspondences extend to the local model diagrams above
(acting trivially on $\Mbl(\La)$). 

\item[2)] The projective limit
\[
\mathscr{S}^{\rm reg}_{K_p}(G, X)=\varprojlim\nolimits_{K^p}{\mathscr{S}}_{K^pK_p}(G, X)
\]
 satisfies the ``dvr extension property":  For every dvr $R$ of mixed characteristic $(0, p)$ we have:
 \[
 \mathscr{S}^{\rm reg}_{K_p}(G, X)(R)={\rm Sh}_{K^p}(G, X)(R[1/p]).
 \]
\end{itemize}
\end{thm}
 
 Note that (a) and (b) together amount to the existence of a smooth morphism
 \[
\bar q_K: \mathscr S^{\rm reg}_K(G, X)\to [\calG\backslash \Mbl(\La)]
 \]
 where the target is the quotient algebraic stack.
 
 \begin{proof}  
 By \cite[Theorem 4.2.7]{KP}, there are schemes $\mathscr S_K(G, X)$ which satisfy similar properties, excluding (c), but 
 with $\Mbl(\La)$ replaced by the PZ local model $\Mloc(\La)$. (Note that all the assumptions of \cite[Theorem 4.2.7]{KP}
 are satisfied: $(G, X)$ is of Hodge type, $p$ is odd, the group $G$ splits over a tamely ramified extension of $\Q_p$, and, 
 by the discussion in \S \ref{ss24}, the
 stabilizer group $\calG_x$ is connected.) In particular, we have
 \begin{equation}\label{LMdiagram2}
\begin{tikzcd}
&\wti{\mathscr{S}}_K(G, X)\arrow[dl, "\pi_K"']\arrow[dr, "q_K"]  & \\
\mathscr S_K(G, X)  &&  \Mloc(\La)
\end{tikzcd}
\end{equation}
with $\pi_K$ a $\calG$-torsor and $q_K$ smooth and $\calG$-equivariant. We set
\[
\wti{\mathscr S}^{\rm reg}_K(G, X)=\wti{\mathscr S}_K(G, X)\times_{\Mloc(\La)}\Mbl(\La)
\]
which carries a diagonal $\calG$-action. Since $r: \Mbl(\La)\tol \Mloc(\La)$ is given by a blow-up, is projective, and we can see  (\cite[\S 2]{PaUnitary}) that the quotient
\[
\pi^{\rm reg}_K: \wti{\mathscr S}^{\rm reg}_K(G, X)\tol {\mathscr S}^{\rm reg}_K(G, X):=\calG\backslash \wti{\mathscr S}^{\rm reg}_K(G, X)
\]
is represented by a scheme and gives a $\calG$-torsor. (This is an example of a ``linear modification", see \cite[\S 2]{PaUnitary}.) 
In fact, since blowing-up commutes with \'etale localization, ${\mathscr S}^{\rm reg}_K(G, X)$ is the blow-up of ${\mathscr S}_K(G, X)$ at the subscheme of closed points that correspond to $*\in \Mloc(\La)$ under the local model diagram (\ref{LMdiagram}). This set of points is the discrete Kottwitz-Rapoport stratum of the special fiber of 
${\mathscr S}_K(G, X)$. The projection gives a smooth $\calG$-morphism
\[
q^{\rm reg}_K: \wti{\mathscr S}^{\rm reg}_K(G, X)\tol \Mbl(\La)
\]
which completes the local model diagram. Property (c) follows from Theorem \ref{mainthm}, Proposition \ref{prop331} and its proof, Proposition \ref{propAdditional}, and properties (a) and (b) which imply that $ {\mathscr S}^{\rm reg}_K(G, X)$ and
$\Mbl(\La)$ are locally isomorphic for the \'etale topology. The rest of the properties in the statement
follow from the corresponding properties for $\mathscr{S}_K(G, X)$ and the construction.
\end{proof} 

\begin{Remark}\label{remShimura}
{\rm a)  As we see in the proof, ${\mathscr S}^{\rm reg}_K(G, X)$ is the blow-up of the discrete
Kottwitz-Rapoport stratum of ${\mathscr S}_K(G, X)$. The geometric
fibers of the blow-up morphism 
\[
{\mathscr S}^{\rm reg}_K(G, X)\to {\mathscr S}_K(G, X)
\]
 over points in this stratum are each isomorphic to $\bbP^{\delta-1}\times \bbP^{d-\delta-1}$. These 
fibers are exactly  the components 
of multiplicity two in the geometric special fiber of $\mathscr S^{\rm reg}_K(G, X)$.

b) When $\delta_*=1$, $\Mloc(\La)$ is already regular.
Hence, in this case ${\mathscr S}_K(G, X)$ is also regular: This 
integral model has appeared, via a different construction, in \cite{MP}.
\quash{
c) The integral models $\mathscr S_K(G, X)$  of \cite{KP} are ``canonical" in the sense that they 
satisfy a characterization property (see \cite{PaCan}). This of course fails for the models $\mathscr S^{\rm reg}_K(G, X)$ constructed here.
However, these have the advantage of having better scheme theoretic properties.}
}
\end{Remark}

 \subsection{Orthogonal integral models}\label{ss62b}

Let us mention that a result exactly like Theorem \ref{LM} can be obtained for the Shimura varieties associated to the
Shimura data $(\SO(V), X)$ and the parahoric subgroup given by the connected stabilizer
\[
K_p=\{g\in \SO(V\otimes_\Q\Q_p)\ |\ g\La=\La,\ \varepsilon(g)=1\}
\]
by combining the previous results with \cite[Theorem 4.6.23]{KP}: Note that $(\SO(V), X)$ is of abelian type. The corresponding 
integral model $\mathscr S_K(\SO(V), X)$ is obtained as a quotient of the integral model $\mathscr S_K(\GSpin(V), X)$ by a finite group action (\cite[\S 4.6]{KP}). We can then construct a regular integral model $\mathscr S^{\rm reg}_K(\SO(V), X)$ with $\Mbl(\La)$ as its local model by following the argument in the proof of Theorem \ref{LM} above.

 \subsection{Rapoport-Zink spaces}\label{ssRZ}
 
 Finally, we observe that our result can be applied to construct regular formal models of certain related Rapoport-Zink spaces.
 Our discussion will be brief, since passing from integral models of Shimura varieties to corresponding Rapoport-Zink formal schemes
  (which can be thought of as integral models of {\sl local} Shimura varieties) is, for the most part, routine. See for example \cite[\S 4]{HPR} 
 for another instance of this parallel treatment. 
 
We consider a local Shimura datum $(\GSpin(V), b, \{\mu\})$, with $V$ over $\Q_p$ and $\mu$ as above and fix the level subgroup $K$ to be the stabilizer of a vertex lattice $\Lambda$. Then a ``Rapoport-Zink formal scheme" ${\Mfr}^b:={\Mfr}_{(\GSpin(V),\mu, b, K)}$ over ${\rm Spf}(\Z_p)$ is constructed in \cite[\S 5]{HamaKim}, provided that $b$ is basic or $\GSpin(V)$ is residually split. (By work of R. Zhou \cite[Proposition 6.5]{Zhou}, this assumption implies that Axiom (A) of \cite[5.3]{HamaKim} is satisfied so the construction in \emph{loc. cit.} applies, but it should not be necessary. 
 The group is residually split when $(V, \lan\ ,\ \ran)$ affords a basis as in one of the four cases of \S \ref{ss42a}.)
In what follows, we assume that $b$ is basic.
 Then the formal scheme $\Mfr^b$ uniformizes the formal completion of the integral model ${\mathscr S}_K(G, X)$ of a corresponding Shimura variety ${\mathrm {Sh}}_K(G, X)$ along the basic locus of its special fiber, see \cite{HamaKim}, \cite{Oki}. 
 
 By its construction, $\Mfr^b$ supports a local model diagram
 \begin{equation}\label{LMdiagram3}
\begin{tikzcd}
&\wti{\Mfr^b}\arrow[dl, "\hat\pi"']\arrow[dr, "\hat q"]  & \\
\Mfr^b &&  \widehat{\Mloc(\La)}
\end{tikzcd}
\end{equation}
 of formal schemes over ${\rm Spf}(\Z_p)$. (In this, $\hat\pi$ is a $\calG$-torsor, $\hat q$ is formally smooth, and $\widehat{\Mloc(\La)}$ denotes the formal $p$-adic completion of 
 $\Mloc(\La)$, see \cite{Oki} for details.) Our results now imply that the blow-up $\Mfr^{b, \rm reg}$ of $\Mfr^{b}$ along the discrete stratum is regular and has the same \'etale local structure as described for ${\mathscr S}^{\rm reg}_K(G, X)$ in Theorem \ref{LM} (c). In fact, we can see that $\Mfr^{b, \rm reg}$ can be used to uniformize the completion of
 ${\mathscr S}^{\rm reg}_K(G, X)$ along its basic locus and that it affords a diagram
  \begin{equation}\label{LMdiagram4}
\begin{tikzcd}
&\wti{\Mfr}^{b, \rm reg} \arrow[dl, "\hat\pi^{\rm reg}"']\arrow[dr, "\hat q^{\rm reg}"]  & \\
\Mfr^{b, \rm reg} &&  \widehat{\Mbl(\La)}
\end{tikzcd}
\end{equation}
with similar properties as above.


\begin{thebibliography}{A99}
 
  \bibitem{FGHM} F. Andreatta, E. Goren, B. Howard, K. Madapusi Pera, \textit{Faltings Heights of Abelian Varieties with Complex Multiplication.} Annals of Math. 187, no. 2 (2018), 391-531.
 
 \bibitem{BTII} F. Bruhat, J. Tits, \textit{Groupes réductifs sur un corps local : II. Schémas en groupes. Existence d'une donnée radicielle valuée}. Publ. Math. IHES 60 (1984).
 
  \bibitem{BTIV} F. Bruhat, J. Tits, \textit{Sch\'emas en groupes et immeubles des groupes classiques sur un corps local. II. Groupes unitaires.} 
Bull. Soc. Math. France 115 (1987), no. 2, 141--195. 
  
\bibitem{CompleteQ} C. De Concini, M. Goresky, R. MacPherson, C. Procesi,
\textit{On the geometry of quadrics and their degenerations. }
Comment. Math. Helv. 63 (1988), no. 3, 337--413. 

\bibitem{DeligneCorvallis}
P.~Deligne, \emph{Vari\'et\'es de {S}himura: interpr\'etation modulaire, et
  techniques de construction de mod\`eles canoniques}, Automorphic forms,
  representations and {$L$}-functions,   {P}art 2, Proc. Sympos.
  Pure Math., XXXIII, Amer. Math. Soc., Providence, R.I., 1979, pp.~247--289.
 
 

  \bibitem{Faltings} G. Faltings, \textit{The category ${\mathcal {MF}}$ in the semistable case}, Izv. Ross. Akad. Nauk Ser. Mat. 80 (2016), no. 5, p. 41--60.
  
\bibitem{HamaKim} P.  Hamacher, W. Kim, \textit{$\ell$-adic \'etale cohomology of Shimura varieties of Hodge type with non-trivial coefficients.}
Math. Ann. 375 (2019), no. 3-4, 973--1044. 
 
 \bibitem{HPR} X. He, G. Pappas, M. Rapoport,
\emph{Good and semi-stable reductions of Shimura varieties}, 
Journal de l'\'Ecole polytechnique - Math., Vol. 7  (2020), p. 497--571.

  
\bibitem{KisinJAMS}
M.~Kisin, \emph{Integral models for {S}himura varieties of abelian type}, J.
  Amer. Math. Soc. {23} (2010), no.~4, 967--1012.

 
\bibitem{KP} M. Kisin, G. Pappas, \textit{Integral models of Shimura varieties with parahoric level structure},   Publ. Math. IHES  128 (2018), 121--218. 

 \bibitem{La}  E. Landvogt, Some functorial properties of the Bruhat-Tits building, J. Reine Angew. Math., 518 (2000), 213--241.
 
 \bibitem{MP} K. Madapusi Pera, \textit{Integral canonical models for spin Shimura varieties.} Compos. Math. 152 (2016), no. 4, 769--824.
 
\bibitem{Oki} Y. Oki, \textit{Notes of Hodge type Rapoport-Zink spaces with parahoric level structure.} arXiv:2012.07076.

 
 \bibitem{PaUnitary} G. Pappas, \textit{On the arithmetic moduli schemes of PEL Shimura varieties.} J. Algebraic Geom. 9 (2000), no. 3, 577--605.

\bibitem{Pappasicm} G. Pappas, \textit{Arithmetic models for Shimura varieties}, 
Proceedings of the ICM -- Rio 2018. Vol. II. Invited lectures, 377--398, World Sci. Publ., Hackensack, NJ, 2018.  

\bibitem{PaCan} G. Pappas, \textit{On integral models of Shimura varieties.} arXiv:2003.13040.

  \bibitem{PR} G. Pappas, M. Rapoport, \textit{Twisted loop groups and their affine flag varieties}. With an appendix by T. Haines and Rapoport. Advances in
Math. 219 (2008), no. 1, 118-198.

  \bibitem{PRS} G. Pappas, M. Rapoport, B. Smithling, \emph{Local models of Shimura varieties, I. Geometry and combinatorics.} Handbook of moduli. Vol. III, 135--217,  Adv. Lect. Math. (ALM), 26, Int. Press, Somerville, MA, 2013.

\bibitem{PZ} G. Pappas, X. Zhu, \textit{Local models of Shimura varieties and a conjecture of Kottwitz}, Invent. Math. {\bf 194} (2013), 147--254.

 

\bibitem  {RZbook} M.\ Rapoport, T.\ Zink,  \textit{Period spaces for $p$--divisible groups}.
Ann.\ of Math. Studies {141}, Princeton University Press, Princeton, NJ, 1996. 



 
 
 \bibitem{Schber} P. Scholze, J. Weinstein, \textit{Berkeley lectures on $p$-adic geometry}, Annals of Mathematics Studies, 207,
Princeton University Press, Princeton, 2020. 
 
\bibitem{Tits} J.~Tits,  \textit{Reductive groups over local fields}, Automorphic forms, representations and L-functions, Proc. Sympos. Pure Math., XXXIII,  part 1, 29--69, Amer. Math. Soc., Providence, R.I., 1979.


  
 
 \bibitem{Zachos} I. Zachos, \textit{On orthogonal local models of Hodge type}, arXiv:2006.07271. 
 
  \bibitem{Zhou} R. Zhou, \textit{Mod-$p$ isogeny classes on Shimura varieties with parahoric level structure.} Duke. Math. J. 169
(2020), no. 15, 2937--3031.

   \end{thebibliography}
\end{document}